\documentclass[reqno,12pt]{amsart} %amsproc/amsbook
\usepackage{amssymb,amsmath,mathrsfs,bbm,color}
\usepackage[colorlinks=true, pdfstartview=FitV, linkcolor=blue, citecolor=red, urlcolor=blue,pagebackref=false]{hyperref}
\usepackage[font=footnotesize]{caption}
\usepackage{tikz}
\usepackage{float}
\usetikzlibrary{arrows.meta}
\usetikzlibrary{decorations.markings}

\usepackage{esint}
\usepackage{bookmark}
\usepackage[left=1in, right=1in, bottom=1.5in]{geometry} %top=1.5in,
\newtheorem{theorem}{Theorem}[section]
\newtheorem{lemma}[theorem]{Lemma}

\newtheorem{corollary}[theorem]{Corollary}
\theoremstyle{definition}

\theoremstyle{remark}
\newtheorem{remark}[theorem]{Remark}

\numberwithin{equation}{section}

\newcommand{\norm}[1]{\lVert#1\rVert}

\newcommand{\cL}{\mathcal{L}}

\newcommand{\R}{\mathbb{R}}

\newcommand{\e}{\varepsilon}
\newcommand{\txt}[1]{\text{#1}}

\begin{document}

\title[]{Doubling inequalities and nodal sets in periodic elliptic homogenization}


%    author one information
\author {Carlos E. Kenig}
\address{
	Department of Mathematics\\
	University of Chicago\\
	Chicago, IL 60637, USA\\
	Email:  cek@math.uchicago.edu }

\author{ Jiuyi Zhu}
\address{
Department of Mathematics\\
Louisiana State University\\
Baton Rouge, LA 70803, USA\\
Email:  zhu@math.lsu.edu }
\author {Jinping Zhuge}
\address{
Department of Mathematics\\
University of Chicago\\
Chicago, IL 60637, USA\\
Email:   jpzhuge@math.uchicago.edu }

\thanks{Kenig is supported in part by NSF grant DMS-1800082; Zhu is supported in part by NSF grant OIA-1832961}
\date{\today}

%\subjclass[2010]{35B27, 74Q05, 35J57.}
%    The 2010 edition of the Mathematics Subject Classification is
%    now available.  If you are citing a classification from the
%    new scheme, use the following input coding instead.

%\begin{abstract}
%	
%\end{abstract}
%\keywords{}
%\commt{In order to distinguish from Lin-Shen's paper, what if adding ``quantitative'' before ``doubling inequalities'' in the title?}
\begin{abstract}
	We prove explicit doubling inequalities and obtain uniform upper bounds (under $(d-1)$-dimensional Hausdorff measure) of  nodal sets of weak solutions for a family of linear elliptic equations with rapidly  oscillating periodic coefficients. The doubling inequalities, explicitly depending on the doubling index, are proved at different scales by a combination of convergence rates, a three-ball inequality from certain ``analyticity'', and a monotonicity formula of a frequency function. The upper bounds of nodal sets are shown by using the doubling inequalities, approximations by harmonic functions and an iteration argument.
%We investigate the explicit doubling inequalities and upper bounds of nodal sets of solutions in  periodic elliptic homogenization. The explicit doubling inequalities are proved, at different scales, by a combination of convergence rates, a three-ball inequality from large-scale analyticity, and a monotonicity formula of frequency function. The explicit upper bounds of nodal sets are shown by using doubling inequalities, approximation of harmonic functions and an iteration argument.
\end{abstract}

\maketitle
\section{Introduction}

The paper is concerned with doubling inequalities and upper bounds of nodal sets of solutions in periodic elliptic homogenization.
We consider a family of elliptic operators in divergence form with rapidly oscillating periodic coefficients
\begin{align}\label{main}
\mathcal{L}_\e =-\nabla \cdot \big(A(x/\e)\nabla \big),
\end{align}
where $\e>0$, and $A(y)=\big( a_{ij}(y)\big)$ is a symmetric $d\times d$ matrix-valued function in $\mathbb R^d$ with dimension $d\geq 2$. Assume that $A(y)$ satisfies the following assumptions:

\begin{itemize}
	\item Strong ellipticity: there is $\Lambda >0$ such that
	\begin{align}\label{ellip}
	 \Lambda |\xi|^2\leq \langle A(y)\xi, \ \xi \rangle\leq  |\xi|^2, \quad \quad \txt{for any } y\in \R^d, \xi\in \R^d.
	\end{align}
	
	\item Periodicity:
	\begin{equation}\label{perio}
	A(y+z)=A(y) \quad \quad \mbox{for any } \ y\in  \mathbb R^d \ \mbox{and} \ z\in \mathbb Z^d.
	\end{equation}
	
	\item Lipschitz continuity: There exists a constant $\gamma \ge 0$ such that
	\begin{align}\label{cond.Lip}
	|A(x)-A(y)|\le \gamma|x-y|, \qquad \txt{for any }x, y\in \mathbb R^d.
	\end{align}
\end{itemize}

The doubling inequality describes quantitative behavior to characterize the strong unique continuation property, which has important applications in inverse problems, control theory and the study of nodal sets of eigenfunctions. For harmonic functions or solutions of general elliptic equations in divergence form with Lipschitz coefficients, the doubling inequality is a consequence of a monotonicity formula or Carleman estimates; see \cite{GL86,JL96,K07,Z16}.  %The classical results for doubling inequalities using Carleman estimates or frequency functions.
In periodic elliptic homogenization, the first doubling inequality was obtained recently by Lin and Shen \cite{LS19} with an implicit dependence on the doubling index. Precisely, they proved that if $u_\e$ is a weak solution of $\cL_{\e}(u_\e) = 0$ in $B_2 = B_2(0)$ and
\begin{equation}\label{precond}
\int_{B_2} u_\e^2 \le N \int_{B_{\Lambda}} u_\e^2,
\end{equation}
then for any $r\in (0,1)$,
\begin{equation}\label{est.LinShen}
\int_{B_{r}} u_\e^2 \le C(N) \int_{B_{r/2}} u_\e^2,
\end{equation}
where $C(N)$ depends only on $d,\Lambda,\gamma$ and $N$. The point here is that the constant $C(N)$ is independent of the small parameter $\e$. This cannot be derived directly from the classical doubling inequality as the Lipschitz constant of the coefficients blows up as $\e$ approaches zero. However, it is not known that how the constant $C(N)$ in (\ref{est.LinShen}) depends on $N$, because (\ref{est.LinShen}) was proved by a compactness argument. We mention that if $\e = 1$, the classical doubling inequality shows that $C(N) = CN^K$ for some $C,K\ge 1$; also see Lemma \ref{lem.Doubling.small}.

On the other hand, the Hadamard three-ball inequality also describes the quantitative unique continuation property. In periodic elliptic homogenization, two different versions of the three-ball inequality with error terms were discovered in \cite{AKS20} and \cite{KZ19}. In general, the three-ball inequalities with errors are weaker than the doubling inequalities, as they alone do not imply the strong unique continuation.

Our first goal of this paper is to find an explicit estimate for the constant $C(N)$ in the doubling inequality in periodic elliptic homogenization. The explicit doubling inequality not only provides more clear quantitative information for the solutions (such as the vanishing order), but also has more applications. We state the result as follows.

%We first state a conjecture for the doubling inequality. (\textcolor{red}{ I suggest we remove the conjecture predicted by Carlos. If you have strong evidence about the validity of the conjecture, you may list here})
%
%\begin{conjecture}\label{conjecture}
%	There exist constants $C>1, K>1$ and $\theta \in (0,1/2)$, depending only on $d,\Lambda$ and $\gamma$, such that if $u_\e$ is a weak solution of $\cL_\e (u_\e) = 0$ in $B_1$ satisfying
%	\begin{equation}
%	\int_{B_1} u_\e^2 \le N \int_{B_{\theta}} u_\e^2,
%	\end{equation}
%	then for every $r\in (0,1)$,
%	\begin{equation}
%	\int_{B_{r}} u_\e^2 \le CN^K \int_{B_{\theta r}} u_\e^2.
%	\end{equation}
%\end{conjecture}
%
%
%We point out that the conjecture is true if $\e = 1$ (see Lemma \ref{lem.Doubling.small}). At present, we are unable to prove the above conjecture. However, we are able to give an explicit doubling index in dependent of $\e$, which seems to have reached a critical state.

\begin{theorem}\label{thm.main}
Assume that $A=A(y)$ satisfies the conditions  (\ref{ellip}), (\ref{perio}) and (\ref{cond.Lip}). Let $u_\e$ be a weak solution of $\cL_\e (u_\e) = 0$ in $B_1$.
\begin{itemize}
	\item[(i)] For $d \ge 3$ and every $\tau>0$, there exist $\theta\in (0,1/2)$ and $C>1$, depending only on $d, \tau, \Lambda$ and $\gamma$, such that if $u_\e$ satisfies
		\begin{equation}\label{est.B1}
		\int_{B_1} u_\e^2 \le N \int_{B_{\theta}} u_\e^2,
	\end{equation}
	then for every $r \in (0,1)$,
	\begin{equation}\label{est.Br}
		\int_{B_{r}} u_\e^2 \le \exp(\exp(CN^\tau)) \int_{B_{\theta r}} u_\e^2.
	\end{equation}
	\item[(ii)] For $d = 2$, there exists a constant $C>1$ depending only on $\Lambda$ and $\gamma$ such that if $u_\e$ satisfies
	\begin{equation}\label{est.B1-2}
		\int_{B_1} u_\e^2 \le N \int_{B_{\frac{\Lambda}{2}}} u_\e^2,
	\end{equation}
	then for every $r\in (0,1)$,
	\begin{equation}\label{est.Br-2}
		\int_{B_{r}} u_\e^2 \le \exp(C(\ln N)^2) \int_{B_{\frac{r}{2}}} u_\e^2.
	\end{equation}
\end{itemize}
\end{theorem}

%\begin{remark}
%	Note that $\beta>0$ in the above theorem may be arbitrarily small. If we can replace $N^\beta$ by $\ln\ln N$, then the conjecture is proved.
%\end{remark}

The double exponential growth $\exp(\exp(CN^\tau))$ for $d\geq 3$ in (\ref{est.Br}) and sub-exponential growth $\exp(C(\ln N)^2) = N^{C\ln N}$ for $d=2$ in (\ref{est.Br-2}) seem to be the best we can obtain from the existing tools and results; see Remark \ref{re-three}. Our ultimate hope is for an estimate of the form $C(N)=N^{{C(\ln\ln N)}^p}$, for some $p>0$ depending on $d, \Lambda$ and $\gamma$. Such an estimate would have very important consequences for the study of long-standing open problems regarding the spectral properties of second order elliptic operators with periodic coefficients and their quantitative unique continuation properties (see for instance Conjecture 6.13, Theorem 6.15 and Conjecture 6.16 in \cite{K16}). This connection between the conjectured optimal doubling estimates and Conjecture 6.13 in \cite{K16} was observed by the first author, D. Mendelson and C. Smart in the fall of 2019. This motivated the current work.

As a straightforward corollary, Theorem \ref{thm.main} implies that the vanishing order of $u_\e$ at the origin does not exceed  $\exp(CN^\tau)$ for $d\geq 3$ and $C(\ln N)^2$ for $d=2$.
Theorem \ref{thm.main} also implies a three-ball inequality without an error term, in contrast to the results in \cite{AKS20} and \cite{KZ19}, namely (e.g., for $d\ge 3$),
\begin{align}
\int_{B_{\theta r}} u_\e^2 \le \exp(\exp(CN^\tau)) \bigg(\int_{B_{\theta^2 r}} u_\e^2\bigg)^{\tau_1} \bigg(\int_{B_{ r}} u_\e^2\bigg)^{1-\tau_1}
\end{align}
for any $0<\tau_1<1$.

The proof of Theorem \ref{thm.main} breaks down into three steps:
\begin{itemize}
	\item Step 1: $\e/r \lesssim N^{-5}$. In this case, we take advantage of the convergence rate in homogenization theory and use the precise three-ball inequality of harmonic functions. The smoothness of the coefficients is not needed in this step.

	\item Step 2: In this step, we need to use ``analyticity'', which distinguishes between $d\ge 3$ and $d=2$. For $d\ge 3$, we let $N^{-5} \lesssim \e/r \lesssim N^{-\frac{1}{2}\tau}$, and use a three-ball inequality with a sharp exponential error term proved recently in \cite{AKS20} by Armstrong, Kuusi and Smart, which is a consequence of the ``large-scale analyticity'' from periodic homogenization. This will lead to a nontrivial improvement on the exponent so that $\tau>0$ in Theorem \ref{thm.main} can be arbitrarily small. Again in this case, the periodic structure will play a role; but the smoothness of coefficients is still not required. For $d=2$, we let $N^{-5} \lesssim \e/r \lesssim 1$, and apply a doubling inequality derived from quasi-regular mappings \cite{AL08} (related to complex analyticity), which requires no smoothness or periodicity on the coefficients. Unfortunately, this method works only in two dimensions.

	\item Step 3: $\e/r \gtrsim N^{-\frac{1}{2}\tau}$ for $d\ge 3$ or $\e/r \gtrsim 1$ for $d=2$. In this case, the classical doubling inequality for elliptic operators with Lipschitz coefficients can be handled by a monotonicity formula for the frequency function. If $d\geq 3$, the Lipschitz constant of the coefficients turns out to be $O(N^{\frac{1}{2}\tau})$ after rescaling. A careful calculation shows that the constant in (\ref{est.Br}) is at least $ \exp(\exp(CN^\tau))$, if the periodicity is not used. If $d=2$, the Lipschitz constant of the coefficients after rescaling is bounded by $C$, independent of $\e$ and $N$. This allows us to obtain a much better estimate in two dimensions.
\end{itemize}

For $d\ge 3$, one will see in the proof that the estimate in Step 3 leads to the double-exponential growth of the constant in (\ref{est.Br}).
What happens when $\e/r \gtrsim N^{-\frac12 \tau}$? To gain some intuition, consider a typical harmonic function $w_k = r^k \cos(k\alpha)$ in $\R^2$ (see \cite{GL86} or \cite{HL13}). Note that
\begin{equation*}
\frac{\int_{B_1} w_k^2}{\int_{B_\theta} w_k^2} = C\theta^{-2k}.
\end{equation*}
By setting $N = C\theta^{-2k}$, we see that the intrinsic frequency of $w_k$ (i.e., the number of times that $w_k$ changes signs) is approximately $\ln N/(-\ln \theta)$. Now, let $u_\e$ be a weak solution of $\cL_{\e}(u_\e) = 0$ whose limit is $w_k$ (the homogenized solution) as $\e \to 0$. In view of the interior first-order approximation $u_\e \sim w_k + \e \chi(x/\e) \nabla w_k$,
the intrinsic frequency of $w_k$ will interact with the frequency of oscillation of the corrector $\chi(x/\e)$. Particularly, under rescaling, if $r/\e \approx \ln N$, the frequency of oscillation of the rescaled coefficients $A(rx/\e)$ (or correctors) is comparable to the intrinsic frequency of $w_k$. Note that the intrinsic frequency does not change under rescaling. It seems that the resonance between these two frequencies causes the failure of the arguments in Step 1 and Step 2 when ${\e/r} \approx (\ln N)^{-1} \gtrsim N^{-\frac12 \tau}$ (note that $\tau$ can be arbitrarily small and thus $N^{-\frac12 \tau}$ is close to the resonant situation), and we do not have a tool to handle this situation (except for $d=2$). We believe that an effective argument should take advantage of both the periodicity and the Lipschitz continuity of the coefficients. %\textcolor{blue}{(I do not quite understand these sayings. I do not know why it can be assumed that $r/\e \approx \ln N$ because this is the guessed frequency for some harmonic functions. Does it comes from the issue on $\hat{\tau}=1/2$? It seems to me that these long sayings will confuse readers. Maybe we do not have to show readers these hard stuff. Maybe we can make it much concise. The following saying is interesting: We believe that an effective argument should take advantage of both the periodicity and the Lipschitz continuity of the coefficients.}

%\textcolor{red}{How do you get this?}
%\textcolor{blue}{ Under the conditions in Theorem \ref{thm.main}, by the doubling inequality, it holds that
%\begin{align}
%(\int_{B_{\theta r}} u_\e^2)^\alpha \le \exp(\exp(CN^\beta)) \big(\int_{B_{\theta^2 r}} u_\e^2\big)^{\alpha}
%\end{align}
%and the following is trivial
%\begin{align}
% (\int_{B_{\theta r}} u_\e^2)^{1-\alpha} \le      \big(\int_{B_{ r}} u_\e^2\big)^{1-\alpha}
%\end{align}
%Together, it will give the three-ball inequality. This is why doubling inequality is stronger than three-ball inequality.
%}

Our second goal is to obtain an upper bound for the nodal sets of solutions in periodic elliptic homogenization.
The study of the $(d-1)$-dimensional Hausdorff measure of nodal sets centers
around  Yau's conjecture for Laplace eigenfunctions on smooth manifolds:
\begin{equation}
-\Delta_g \phi_\lambda=\lambda \phi_\lambda, \qquad \txt{on } \mathcal{M},
\label{class}\end{equation}
where $\mathcal{M}$ is a compact smooth Riemannian manifold without boundary.
It was conjectured in \cite{Y82} that the bounds of nodal sets of
eigenfunctions in (\ref{class}) are controlled by
\begin{equation}c\sqrt{\lambda}\leq H^{d-1}(\{x\in\mathcal{M}|\phi_\lambda(x)=0\})\leq C\sqrt{\lambda}\label{yau}
\end{equation} where $C, c$
depend only on the manifold $\mathcal{M}$ and $H^{d-1}$ denotes the $(d-1)$-dimensional Hausdorff measure. The conjecture (\ref{yau}) was shown
for real analytic manifolds by Donnelly-Fefferman in
\cite{DF88}. Lin \cite{L91} also proved the upper bound for the
analytic case, using an approach by frequency functions. We should mention that, by a lifting argument, Yau's conjecture can be reduced to studying the nodal sets of harmonic functions on smooth manifolds. In recent years, there was an important breakthrough made by Logunov and Malinnikova \cite{LM18-1},  \cite{L18-1} and \cite{L18-2}. A polynomial upper bound was given in \cite{L18-1} and the sharp lower bound in the conjecture was shown in \cite{L18-2}. We are interested in the upper bound of nodal sets for $\cL_{\e}(u_\e) = 0$ with rapidly oscillating periodic coefficients. The study of nodal sets in homogenization was initiated by Lin and Shen \cite{LS19}, where an implicit upper bound depending on the doubling index was shown. We are able to provide an explicit upper bound.

\begin{theorem}\label{mainth}
	Assume that $A=A(y)$ satisfies the conditions  (\ref{ellip}), (\ref{perio}), and (\ref{cond.Lip}). Let $u_\e$ be a nonzero weak solution of $\mathcal{L}_\e (u_\e)=0$ satisfying (\ref{precond}).
	\begin{itemize}
		\item[(i)] If $d\ge 3$, then for any $\alpha>8$, it holds that
		\begin{align}\label{est.Nodal}
			H^{d-1}(\{x\in B_\frac{\Lambda}{16}| u_\e(x)=0\})\leq \exp(CN^{\alpha}),
		\end{align}
		where $C$ depends only on $d$, $\Lambda, \gamma$ and $\alpha$.
		\item[(ii)] If $d = 2$, then it holds that
		\begin{align}\label{est.Nodal2d}
			H^{1}(\{x\in B_\frac{\Lambda}{16}| u_\e(x)=0\})\leq \exp(C(\ln N)^2),
		\end{align}
		where $C$ depends only on $\Lambda$ and $\gamma$.
	\end{itemize}
\end{theorem}

The strategy of the proof is as follows. For relatively large $\e$, we adapt a blow-up argument to obtain the upper bounds of nodal sets. For small $\e$, the solution $u_\e$ can be approximated by a harmonic function $u_0$, and thus the nodal set of $u_\e$ is a small perturbation of the nodal set of $u_0$. We then derive a quantitative estimate for the nodal set of $u_\e$ by carefully studying the small perturbations near the nodal set and critical set of $u_0$, which has its root in the analogous qualitative estimates obtained in \cite{LS19}. By iterating such quantitative estimate, we are able to show the upper bound for the nodal sets of $u_\e$. The restriction $\alpha>8$ for $d\ge 3$ arises from the doubling inequality (\ref{generaldouble}) for $\beta\in(\frac{3}{4}, 1)$. If we consider $N$ to be $\exp(CM)$ for some large constant $M$, which is the case for the doubling inequality of eigenfunctions, the upper bounds of nodal sets are double exponential functions  $\exp(\exp(CM))$. In this sense, the restriction $\alpha>8$ only affects the constant $C$ in such upper bounds, which does not play an important role. For $d = 2$, we point out that there is no misprint in the exponential (compared to $d\ge 3$). We still have the exponential, because in this situation, instead of the doubling inequality in (\ref{est.Br-2}), the suboptimal quantitative stratification of the critical set of harmonic functions \cite{NV17} dominates the upper bound.

%\textcolor{red}{Explain the ideas of the proof; explain the restriction $\alpha > 8$.}\textcolor{blue}{Please see above.}

The paper is organized as follows. Section 2 is devoted to a doubling inequality at relatively large scales by the homogenization theory. In section 3, we derive the doubling inequality, using frequency functions, and show how it depends on the large Lipschitz constant of the coefficients. Then, Theorem \ref{thm.main} is proved in section 4 and Theorem \ref{mainth} is proved in section 5.
 Throughout the paper, the letters $c$, $C$, $\hat{C}$, $\tilde{C}$, $C_i$, $c_i$ denote  positive
constants that do not depend on $\e$ or $u_\e$, and they may vary from line to line.

\textbf{Acknowledgements:}
Parts of this work were carried out during the second author's visit to the Department of Mathematics at the University of Chicago during January-March 2020. The first author would like to thank D. Mendelson and C. Smart for many insightful discussions on doubling inequalities in periodic homogenization. The second author would like to thank the Department of Mathematics at Chicago for the warm hospitality and the wonderful academic atmosphere. The second author also would like to thank F. Lin for helpful discussions on three-ball inequalities in periodic homogenization. The third author would like to thank D. Mendelson for insightful discussions during the early stages of this work.

\section{Homogenization}
In this section, we deal with the case $\e/r \gtrsim N^{-\frac{1}{\beta-\frac{3}{4}}}$ for all dimensions. Indeed, we will prove a quantitative version of \cite[Theorem 3.1]{LS19}.

Let $\cL_0 = -\nabla\cdot (\widehat{A}\nabla)$ be the homogenized operator and $\widehat{A}$ be the homogenized coefficient matrix of $A$ (see, e.g., \cite{S18} for the general theory of periodic elliptic homogenization). Define the ellipsoid
\begin{equation*}
E_r = \big\{ x\in \R^d: \langle(\widehat{A})^{-1} x, x\rangle < r^2 \big\}.
\end{equation*}

The following is the main theorem of this section.

\begin{theorem}\label{thm.large.r}
	Let $\theta\in (0,1/2]$ and $A$ satisfy  conditions  (\ref{ellip}), (\ref{perio}) and (\ref{cond.Lip}). There exists $C>0$ depending only on $d$ and $\Lambda$ such that
	if $\cL_\e (u_\e) = 0$ in $E_1$ and
	\begin{equation*}
	\int_{E_1} u_\e^2 \le N \int_{E_{\theta}} u_\e^2,
	\end{equation*}
	then for any  $CN^{\frac{1}{\beta-\frac{3}{4}}} \e < r<1-\sqrt{\e}$, we have
	\begin{equation*}
	\int_{E_r} u_\e^2 \le 2N \int_{E_{\theta r}} u_\e^2.
	\end{equation*}
\end{theorem}

This follows from Lemma \ref{lem.1step-dual} and Lemma \ref{lem.Ak}.

\begin{lemma}\label{lem.1step-dual}
	Let $\theta\in (0,1/2]$. Suppose $u_\e$ is a solution of $\cL_\e (u_\e) = 0$ in $E_1$ satisfying
	\begin{equation*}
	\int_{E_1} u_\e^2 \le N \int_{E_{\theta}} u_\e^2.
	\end{equation*}
	For any $\beta\in (3/4,1)$, there exist  $c,C>0$, depending only on $d,\Lambda$ and $\beta$, such that if $\e < cN^{-\frac{1}{\beta-\frac34}}$, then for any $r\in [\theta,1-\sqrt{\e}]$
	\begin{equation}\label{est.B1-B.5}
	\int_{E_r} u_\e^2 \le N(1+CN \e^{\beta-\frac34})\int_{E_{\theta r}} u_\e^2.
	\end{equation}
\end{lemma}
\begin{proof}
	Let $t>0$, to be determined. Since $\txt{dist}(\partial E_{1-t}, \partial E_1) \le Ct$, by the Caccioppoli inequality, we have
	\begin{equation*}
	\int_{E_{1-t}} |\nabla u_\e|^2 + \int_{E_{1-t}}  u_\e^2 \le \frac{C}{t^2} \int_{E_1} u_\e^2 \le \frac{CN}{t^2} \int_{E_{\theta}} u_\e^2.
	\end{equation*}
	By the co-area formula, we can find some $c_0 \in (1,2)$ so that
	\begin{equation*}
	\int_{\partial E_{1-c_0t}} |\nabla u_\e|^2 + \int_{\partial E_{1-c_0t}}  u_\e^2 \le \frac{CN}{t^3} \int_{E_{\theta}} u_\e^2.
	\end{equation*}
	Without loss of generality, let us simply assume $c_0 = 1$. Hence $u_\e|_{\partial E_{1-t}} \in H^1(\partial E_{1-t})$. By \cite[Theorem 1.1]{KLS12},
	\begin{equation}\label{est.ue-u0}
	\int_{E_{1-t}} (u_\e - u_0)^2 \le C\e^{2\beta} \norm{u_\e}_{H^1(\partial E_{1-t})}^2 \le \frac{CN\e^{2\beta}}{t^3} \int_{E_{\theta}} u_\e^2 ,
	\end{equation}
	where $u_0$ is the solution of $\cL_0 (u_0) = 0$ and $u_0 = u_\e $ on $\partial E_{1-t}$ and $\beta\in (0,1)$ is arbitrary.
	
	As a result, we have
	\begin{equation}\label{est.u0.B2t}
	\begin{aligned}
	\norm{u_0}_{L^2( E_{1-t})} & \le \norm{u_\e}_{L^2(E_1)} + \norm{u_\e - u_0}_{L^2(E_{1-t})} \\
	& \le \sqrt{N} (1+C\e^\beta t^{-3/2}) \norm{u_\e}_{L^2(E_{\theta})}.
	\end{aligned}
	\end{equation}
	Also,
	\begin{equation*}
	\norm{u_\e}_{L^2(E_{\theta})} \le \norm{u_0}_{L^2(E_{\theta})} + C\sqrt{N}\e^\beta t^{-3/2} \norm{u_\e}_{L^2(E_{\theta})}.
	\end{equation*}
	We will choose $t<1$ so that $C\sqrt{N}\e^\beta t^{-3/2} < 1/2$. Consequently,
	\begin{equation}\label{est.B1.ueu0}
	\norm{u_\e}_{L^2(E_{\theta})} \le \big( 1 - C\sqrt{N}\e^\beta t^{-3/2}\big)^{-1} \norm{u_0}_{L^2(E_{\theta})} \le \big( 1 +C \sqrt{N}\e^\beta t^{-3/2}\big)\norm{u_0}_{L^2(E_{\theta})}.
	\end{equation}
	Inserting this into (\ref{est.u0.B2t}), we have
	\begin{equation}\label{est.u0}
	\begin{aligned}
	\norm{u_0}_{L^2(E_{1-t})} & \le \sqrt{N} \big( 1 + C\sqrt{N}\e^\beta t^{-3/2} \big)(1+C\e^\beta t^{-3/2}) \norm{u_0}_{L^2(E_{\theta})}\\
	& \le \sqrt{N} \big( 1 + C\sqrt{N}\e^\beta t^{-3/2} \big) \norm{u_0}_{L^2(E_{\theta})},
	\end{aligned}
	\end{equation}
	where we have used the simple fact that $(1+a)^2 \le 1+3a$ for $a\in [0,1]$ and enlarged the constant $C$ in the last inequality.
	
	Next, by the interior $L^\infty$ estimate for $\widehat{A}$-harmonic functions, we have
	\begin{equation*}
	\begin{aligned}
	\norm{u_0}_{L^2(E_{\theta})} &\le \norm{u_0}_{L^2(E_{\theta(1-t)})} + \norm{u_0}_{L^2(E_{\theta}\setminus E_{\theta(1-t)})} \\
	& \le \norm{u_0}_{L^2(E_{\theta(1-t)})} + C\sqrt{\theta t} \norm{u_0}_{L^2(E_{1-t})}
	\end{aligned}
	\end{equation*}
	Inserting this into (\ref{est.u0}) and choosing $t$ sufficiently small so that $C\sqrt{N}\sqrt{\theta t} < 1/2$, we obtain
	\begin{equation*}
	\norm{u_0}_{L^2(E_{1-t})} \le \sqrt{N} \big( 1 + \sqrt{N}C\e^\beta t^{-3/2} \big) \big( 1+C\sqrt{N} \sqrt{\theta t} \big) \norm{u_0}_{L^2(E_{\theta(1-t)})}.
	\end{equation*}
	
	Choose $t = \sqrt{\e}$. We arrive at
	\begin{equation}\label{est.Double.u0}
	\norm{u_0}_{L^2(E_{1-t})} \le \sqrt{N}(1+C\sqrt{N}\e^{\beta-3/4}) \norm{u_0}_{L^2(E_{\theta(1-t)})}.
	\end{equation}
	Note that the above calculation goes through only if $\sqrt{N}C\e^\beta t^{-3/2} < 1/2$ and $C\sqrt{N}\sqrt{t} < 1/2$. This implies that we require
	\begin{equation*}
	\e \le C^{-1}N^{\frac{-1}{2(\beta-3/4)}},
	\end{equation*}
	for some large constant $C$.
	
	Recall that $u_0$ is a weak solution of $\cL_0(u_0) = 0$ in $E(1-\sqrt{\e})$. Let $w_0(x) = u_0(\widehat{A}^{\frac12} x)$. Then $\Delta w_0 = 0$ in $B_{1-\sqrt{\e}}$ and (\ref{est.Double.u0}) is equivalent to
	\begin{equation}\label{est.1-sqrte}
		\norm{w_0}_{L^2(B_{1-\sqrt{\e}})} \le \sqrt{N}(1+C\sqrt{N}\e^{\beta-3/4}) \norm{w_0}_{L^2(B_{\theta(1-\sqrt{\e})})}.
	\end{equation}
	
	Now, as a consequence of the well-known three-sphere theorem for harmonic functions,
	\begin{equation}\label{est.phi.u0}
	\varphi(r) = \log_2 \fint_{B_{2^r}} w_0^2
	\end{equation}
	is a convex function in $(-\infty,0]$ and therefore $\varphi(t) - \varphi(t-c)$ is a nondecreasing function in $t$, for any fixed $c>0$. Hence, we obtain from (\ref{est.1-sqrte}) that for any $r\in (0,1-\sqrt{\e})$,
	\begin{equation*}
	\norm{w_0}_{L^2(B_r)} \le \sqrt{N}(1+C\sqrt{N}\e^{\beta-3/4}) \norm{w_0}_{L^2(B_{\theta r})}.
	\end{equation*}
	(The doubling index with $\theta$ is an increasing function of radius.) Again, this is equivalent to
	\begin{equation}\label{est.u0Er}
	\norm{u_0}_{L^2(E_r)} \le \sqrt{N}(1+C\sqrt{N}\e^{\beta-3/4}) \norm{u_0}_{L^2(E_{\theta r})}
	\end{equation}
	for any $r\in (0, 1-\sqrt{\e})$.
	
	Now, let $r\in [\theta,1-\sqrt{\e})$. It follows by (\ref{est.u0Er}) that
	\begin{equation*}
	\begin{aligned}
	\norm{u_\e}_{L^2(E_r)} &\le \norm{u_\e - u_0}_{L^2(E_r)} + \norm{u_0}_{L^2(E_r)} \\
	& \le C \sqrt{N} \e^{\beta-3/4} \norm{u_\e}_{L^2(E_{\theta})} + \sqrt{N}(1+C\sqrt{N}\e^{\beta-3/4}) \norm{u_0}_{L^2(E_{\theta r})}\\
	& \le C \sqrt{N} \e^{\beta-3/4} \norm{u_\e}_{L^2(E_{r})} + \sqrt{N}(1+C\sqrt{N}\e^{\beta-3/4}) \norm{u_0 - u_\e}_{L^2(E_{\theta r})} \\
	& \qquad + \sqrt{N}(1+C\sqrt{N}\e^{\beta-3/4}) \norm{u_\e}_{L^2(E_{\theta r})} \\
	& \le C N \e^{\beta-3/4} (1+C\sqrt{N}\e^{\beta-3/4}) \norm{u_\e}_{L^2(E_r)} + \sqrt{N}(1+C\sqrt{N}\e^{\beta-3/4}) \norm{u_\e}_{L^2(E_{\theta r})},
	\end{aligned}
	\end{equation*}
	where we have used the fact $E_{\theta} \subset E_r$ in the third inequality and (\ref{est.ue-u0}) in the second and last inequalities.
	Assume further that $\e \le cN^{-1/(\beta-3/4)}$. Then
	\begin{equation*}
	\begin{aligned}
	\norm{u_\e}_{L^2(E_r)} &\le \frac{\sqrt{N}(1+C\sqrt{N}\e^{\beta-3/4})}{1-C N\e^{\beta-3/4}}  \norm{u_\e}_{L^2(E_{\theta r})} \\
	& \le  \sqrt{N}(1+CN \e^{\beta-3/4})\norm{u_\e}_{L^2(E_{\theta r})}.
	\end{aligned}\end{equation*}
	This proves the lemma.
\end{proof}

Now, if $\e < cN^{-1/(\beta-3/4)}$, the above lemma allows us to iterate (\ref{est.B1-B.5}) down to $r = c^{-1} N^{1/(\beta-3/4)} \e$. Precisely, if $r = \theta ^{k} > C N^{1/(\beta-3/4)}\e $ and
\begin{equation*}
\int_{E_{r}} u_\e^2 \le A_{k} \int_{E_{\theta r}} u_\e^2,
\end{equation*}
with $A_0 = N$, then
\begin{equation*}
\int_{E_{\theta r}} u_\e^2 \le A_{k+1} \int_{E_{\theta^2 r}} u_\e^2,
\end{equation*}
 where
\begin{equation*}
A_{k+1} = A_{k}(1+CA_{k} (\theta^{-k} \e)^{\beta-3/4}),
\end{equation*}
provided $A_k (\theta^{-k} \e)^{\beta-3/4} < c$.

\begin{lemma}\label{lem.Ak}
	For all $k\le k_0$ with $\theta^{-k_0}\e \simeq c_1 N^{-1/(\beta-3/4)}$ and $c_1>0$ sufficiently small, one has
	\begin{equation*}
	A_k \le 2N.
	\end{equation*}
\end{lemma}
\begin{proof}
	Define $B_k = A_k/N$. Then $B_0 = 1$ and
	\begin{equation*}
	B_{k+1} = B_{k}(1+CB_{k} \theta^{-k(\beta-3/4)} \e^{\beta-3/4} N ) \le B_{k}(1+c_1 C B_{k}\theta^{(k_0-k)(\beta-3/4)}  ).
	\end{equation*}
	It follows that
	\begin{equation*}
	B_{k+1} - B_{k} \le c_1 C B_{k}^2 \delta^{k_0-k},
	\end{equation*}
	where $\delta = \theta^{\beta-3/4}\le (1/2)^{\beta-3/4} < 1$. The above inequality yields
	\begin{equation}\label{est.Bk+1}
	B_{k+1} \le B_0 + \sum_{j=0}^{k} c_1 C B_{j}^2 \delta^{k_0-j}.
	\end{equation}
	We prove by induction that if $c_1$ is sufficiently small, then $B_k\le 2$ and $A_k (\theta^{-k} \e)^{\beta-3/4}\le 2c_1$ for all $k\le k_0$. Actually, if
	\begin{equation*}
	c_1 \le (4C \sum_{j=0}^\infty \delta^{j})^{-1} = \frac{1-\delta}{4C},
	\end{equation*}
	and $B_j\le 2$ for all $1\le j\le k$, then
	it is easy to see from (\ref{est.Bk+1}) that $B_{k+1} \le 2$ and $A_k (\theta^{-k} \e)^{\beta-3/4} \le 2N (\theta^{-k_0} \e)^{\beta-3/4} \le 2c_1$. This proves the desired estimate.
\end{proof}

\begin{remark}
	Observe that in the above proof, the smoothness of the coefficients has not been used explicitly, except for (\ref{est.ue-u0}) by \cite[Theorem 1.1]{KLS12}. But this actually can be replaced by, e.g., \cite[Theorem 1.4]{NX19} with $m=1$, which does not require any smoothness.
\end{remark}

\begin{remark}
	It is not difficult to see that (\ref{est.B1-B.5}) implies the following three-ball inequality with an error term
	\begin{equation}\label{est.new3ball}
	\norm {u_\e}_{L^2(E_{\theta r})} \le \norm {u_\e}_{L^2(E_{r})}^{\frac12} \norm {u_\e}_{L^2(E_{\theta^2 r})}^{\frac12} + C \e^{\frac{1}{2}(\beta-\frac34)} \norm {u_\e}_{L^2(E_{r})},
	\end{equation}
	for any $\theta \in (0,\frac12]$ and $\beta \in (\frac34,1)$. Compared to the three-ball inequalities in \cite{AKS20} (see Theorem \ref{thm.AKS} below) and \cite{KZ19}, our major term on the right-hand side of (\ref{est.new3ball}) is sharp. In particular, if $\e \to 0$, (\ref{est.new3ball}) recovers precisely the three-ball inequality for $\widehat{A}$-harmonic functions.
\end{remark}

\begin{theorem}\label{thm.Br}
	Given arbitrary $\theta\in (0,\Lambda/2]$, there exists $C>0$ depending only on $d$ and $\Lambda$ such that
	if $\cL_\e (u_\e) = 0$ in $B_1$ and
	\begin{equation}\label{est.B1Btheta}
	\int_{B_1} u_\e^2 \le N \int_{B_{\theta}} u_\e^2,
	\end{equation}

	then for any  $CN^{\frac{1}{\beta-\frac{3}{4}}} \e < r<1$, we have
	\begin{equation*}
	\int_{B_r} u_\e^2 \le 8N^3 \int_{B_{\theta r}} u_\e^2.
	\end{equation*}
\end{theorem}
\begin{proof}
	This is deduced from Theorem \ref{thm.large.r} and the fact
	\begin{equation}\label{est.ErBr}
	B_{\sqrt{\Lambda}r} \subset E_r \subset B_{r}.
	\end{equation}
	Indeed, (\ref{est.B1Btheta}) and (\ref{est.ErBr}) imply
	\begin{equation*}
	\int_{E_1} u_\e^2 \le N \int_{{B_{\theta}}} u_\e^2 \le N \int_{{E_{\theta/\sqrt{\Lambda}}}} u_\e^2.
	\end{equation*}
	Note that $\theta\in (0,\Lambda/2]$ implies $\theta/\sqrt{\Lambda} \in (0,\sqrt{\Lambda}/2]\subset (0,1/2]$.
Now, if $r\in (CN^{\frac{1}{\beta-\frac{3}{4}}} \e , 1-\sqrt{\e})$, we may apply Theorem \ref{thm.large.r} (three times) with $\theta' = \theta/\sqrt{\Lambda}$ and obtain
	\begin{equation*}
	\int_{B_r} u_\e^2 \le \int_{E_{r/\sqrt{\Lambda}}} u_\e^2 \le (2N)^3 \int_{E_{\theta^3 \Lambda^{-2} r}} u_\e^2 \le 8N^3 \int_{E_{\theta r}} u_\e^2 \le 8N^3\int_{B_{\theta r}} u_\e^2.
	\end{equation*}
	
	For $r\in [1-\sqrt{\e},1]$ (without loss of generality, assume $\e < 1/4$), we may apply Theorem \ref{thm.large.r} once to obtain
	\begin{equation*}
	\int_{B_r} u_\e^2 \le \int_{B_1} u_\e^2 \le N \int_{B_\theta} u_\e^2 \le 2N^2 \int_{B_{\theta^2 \Lambda^{-1} }} u_\e^2 \le 2N^2 \int_{B_{\theta r }} u_\e^2.
	\end{equation*}
	This ends the proof.
	%See the proof of \cite[Theorem 1.2]{LS19}.
\end{proof}

\section{Dependence on the Lipschitz constant}

In this section, we derive the doubling inequality with a large Lipschitz constant, which will be used in the Step 3 of the proof of Theorem \ref{thm.main}. We aim to show how the Lipschitz character of the coefficients plays a role in quantitative unique continuation, which seems to be largely unexplored. Assume that
\begin{align}
\mathcal{L}_1 (u) =-\nabla \cdot \big(A(x)\nabla u\big)=0,
\label{newelli}
\end{align}
where $A(x)$ satisfies (\ref{ellip}) and
\begin{align}
|A({x})-A({y})|\leq L |x-y|
\label{lip2}
\end{align}
for some large positive constant $L>1$. We emphasize that throughout this section, the constant $C$ will never depend on $L$. Since the $L^\infty$ norm and the $L^2$ norm of $u$ are comparable,
parallel to the assumption (\ref{est.B1}),
we may assume the following
\begin{align}
\|u\|_{L^\infty ( B_1)}\leq M \|u\|_{L^\infty ( B_{\theta})}
\label{assdoub}
\end{align}
for some large constant $M>1$.

%\textcolor{blue}{Question: why do you use $L^\infty$ norm instead of $L^2$ norm? Are they interchangeable? Also, I suggest assuming this doubling inequality with balls from $B_1$ to $B_{\theta}$ with any fixed $\theta \in (0,1/2]$. I feel like you only need to change the proof of Lemma \ref{lem.Doubling.small} a little bit. This suggestion is to guarantee the assumption is consistent throughout the paper.} \textcolor{red}{  $L^\infty$ norm is comparable to $L^2$ norm, e.g. $\|u\|_{L^\infty ( B_{3/4})}\leq C\|u\|_{L^2 ( B_1)}\leq CM \|u\|_{L^2 ( B_{\theta})}\leq C M  \|u\|_{L^\infty ( B_{\theta})}$, I need to use $L^\infty$ norm because of (\ref{remind}). $L^2$ norm will make it a mess. }

%Our aim is to obtain the doubling inequality in every small scale.For the Lipsthitz metric,
In order to define the frequency function later, we need to construct the geodesic polar coordinates. The construction of polar coordinates has been obtained in \cite{AKS62}. We adopt a slightly different construction of the metric from \cite[Chapter 3.1]{HL13}. We follow the construction with an eye on the explicit dependence of the Lipschtiz constant $L$. For $d\geq 3$, we define the Lipschitz metric $\hat{g}=\hat{g}_{ij}(x) d x_i\otimes d x_j$ as follows
\begin{align}
\hat{g}_{ij}({x})=a^{ij}({x}) \det (A({x}))^{\frac{1}{d-2}},
\end{align}
where $a^{ij}({x})$ is the entry of $A^{-1}({x})$. The case $d=2$ will be discussed in Remark \ref{rmk.2d}.
Note that $\hat{g}$ is Lipsthitz continuous and satisfies
\begin{align}
|\hat{g}({x})- \hat{g}({y})|\leq C L |x-y|.
\end{align}

Define
\begin{align}
r^2=r^2(x)=\hat{g}_{ij}(0) x_i x_j
\label{dist}
\end{align}
and
\begin{align*}
\psi(x)=\hat{g}^{kl}({x})\frac{\partial r}{\partial x_k} \frac{\partial r}{\partial x_l}.
\end{align*}
From (\ref{dist}), we can also write
\begin{align*}
\psi(x)=\frac{1}{r^2}\hat{g}^{kl}({x}) \hat{g}_{ik}(0)\hat{g}_{jl}(0)x_i x_j.
\end{align*}
Thus, we can check that $\psi(x)$ is a non-negative Lipschitz function satisfying
\begin{align}
|\psi(x)- \psi(y)|\leq CL |x-y|,
\end{align}
where $C$ depends only on $d$ and $\Lambda$. We introduce a new metric  ${g}={g}_{ij}(x) d x_i\otimes d x_j$ by setting
\begin{align}
{g}_{ij}(x) =\psi(x) \hat{g}_{ij}({x}).
\label{metricha}
\end{align}
We can write the metric $g$ in terms of the intrinsic geodesic polar coordinates $(r, \sigma_1, \cdots, \sigma_{d-1})$,
\begin{align}
g= d r \otimes dr +r^2 b_{ij}(r, \sigma) d \sigma_i \otimes d \sigma_j,
\label{contrupoly}
\end{align}
where $b_{ij}$ satisfies
%\begin{align}
%b_{ij}(0, 0)=\delta_{ij}, \quad \mbox{for} \ i, j=1, \cdots, d-1,
%\label{unknownline}
%\end{align}
%\textcolor{blue}{(For Carlos, the above line (\ref{unknownline}) is from the section 3.1 in \cite{HL13}. I think $b_{ij}(0, 0)$ and $b_{ij}(0, \sigma)$ describe the same point. Here $\sigma$ denotes the angles. However, this estimate is not used in the later estimates. We want to delete it. When checking the paper in \cite{AKS62}, I do not find such estimate.  Other estimates are known).}
%and
\begin{align}
|\partial_r b_{ij}(r, \sigma)|\leq CL, \quad \mbox{for} \ i, j=1, \cdots, d-1,
\label{gamma}
\end{align}
and $C$ depends only on $d$ and $\Lambda$.

The existence of the geodesic polar coordinates $(r, \ \sigma)$ allows us to consider geodesic balls. Denote by $\mathbb B_r$ the geodesic ball in the metric $g$ of radius $r$ and centred at the origin. In particular, from (\ref{dist}) and (\ref{contrupoly}), $r(x)=\sqrt{\hat{g}_{ij}(0) x_i x_j}$ is the geodesic distance from $x$ to the origin in the new metric $g$. Thus, it is conformal to the usual Euclidean ball. For convenience of presentation, we may assume that the geodesic balls coincide with the Euclidean balls, i.e., $\hat{g}_{ij}(0) = \delta_{ij}$.

Let
\begin{align}
\eta(x)=\psi^{-\frac{d-2}{2}}.
\end{align}
Obviously, $\eta(x)$ is a Lipschitz function satisfying
\begin{align}\label{use}
C_1\leq \eta(x)\leq C_2,
\end{align}
where $C_1$ and $C_2$ depend on $d$ and $\Lambda$. In the polar coordinates,
\begin{align}
|\partial_r \eta(r, \sigma)|\leq {C}L.
\label{eta}
\end{align}

In this new metric $g$, the equation (\ref{newelli}) can be written as
\begin{align}
-\nabla_g \cdot (\eta(x) \nabla_g u(x))=0  \quad  \mbox{in} \  B_1.
\label{newequ}
\end{align}
Let
\begin{align}
D(r)=\int_{ B_r} \eta |\nabla_g u|^2 d V_g
\end{align}
and
\begin{align}
H(r)=\int_{ \partial B_r} \eta  u^2 dS_g,
\end{align}
where $dS_g$ represents the area element of  $\partial B_r$ under the metric $g$.
We define the frequency function by
\begin{align}
\mathcal{N}(r)=\frac{ rD(r)}{H(r)}.
\end{align}
For future application, we will also use the notation $\mathcal{N}(p, r)$ to specify the center of the ball $B_r(p)$ in the definition of frequency function.
\begin{lemma}
	Let $u\in H^1( B_1)$ be a nontrivial solution of (\ref{newelli}). There exists a positive constant $C$ depending on $d$ and $\Lambda$ such that
	\begin{align}
	\overline{\mathcal{N}}(r)=\exp({CLr})\mathcal{N}(r)
\label{generalfre}
	\end{align}
	is a non-decreasing function of $r\in (0, 1)$.
\label{frequency}
\end{lemma}
\begin{proof}
	The proof of the lemma is essentially contained in \cite{GL86}. Since we want to show the explicit dependence of the Lipschtiz constant $L$ in the estimates, we sketch the proof by considering the role of $L$. Taking derivative with respect to $r$ for $\mathcal{N}$, we have
	\begin{align}
	\frac{\mathcal{N}'(r)}{\mathcal{N}(r)}=\bigg( \frac{1}{r}+\frac{D'(r)}{D(r)}-\frac{H'(r)}{H(r)}\bigg).
	\end{align}
	In order to prove the lemma, it suffices to show
	\begin{align}
	\frac{1}{r}+\frac{D'(r)}{D(r)}-\frac{H'(r)}{H(r)}\geq -CL.
	\label{need}
	\end{align}
	Thus, we consider the derivatives of $H(r)$ and $D(r)$, respectively. Setting $b(r, \sigma)=|\det (b_{ij}(r, \sigma))|$.
	Note that $dS_g=r^{d-1}\sqrt{ b(r, \sigma)} d\sigma$. We write $H(r)$ as
	\begin{align}
	H(r)=r^{d-1} \int_{\partial B_{1}}\eta(r, \sigma) u^2(r, \sigma)\sqrt{ b(r, \sigma)} d\sigma.
	\end{align}
	Taking derivative with respect to $r$, one has
	\begin{align}
	H'(r)=\frac{d-1}{r}H(r)+\int_{\partial  B_r} \frac{1}{\sqrt{b}}\partial_r( \eta \sqrt{b}) u^2 dS_g +2 \int_{\partial  B_r}\eta  u \partial_r  u d S_g,
	\end{align}
	where $\partial_r  u=\langle \nabla_g u, \frac{x}{r} \rangle$ on $\partial B_r$.
	By (\ref{gamma}), (\ref{use}) and (\ref{eta}), we have
	\begin{align}
	H'(r)=\Big(\frac{d-1}{r}+{O(L)}\Big)   H(r)+2 \int_{\partial  B_r}\eta  u \partial_r  u d S_g.
	\end{align}
	Multiplying both sides of (\ref{newequ}) by $u$ and performing the integration by parts give that
	\begin{align}
	D(r)=\int_{ B_r} \eta |\nabla_g u|^2 d V_g =\int_{\partial  B_r} \eta u \partial_r u d S_g.
	\end{align}
	It follows that
	\begin{align}
	H'(r)=\Big(\frac{d-1}{r}+{O(L)} \Big)   H(r)+2 D(r).
	\label{HHH}
	\end{align}
	
	Similarly, we may compute the derivative of $D(r)$ as in \cite{GL86} and obtain
	\begin{align}
	D'(r)=\Big(\frac{d-2}{r}+{O(L)} \Big)D(r)+2\int_{\partial  B_r} \eta ( \partial_r u)^2 d S_g.
	\label{DDD}
	\end{align}
	Combining the estimates (\ref{HHH}) and (\ref{DDD}), and using the Cauchy-Schwarz inequality, we obtain
	\begin{align*}
	\frac{1}{r}+\frac{D'(r)}{D(r)}-\frac{H'(r)}{H(r)}&={O(L)}+2\frac{\int_{\partial  B_r} \eta ( \partial_r u)^2 d S_g }{\int_{\partial  B_r} \eta u \partial_r u  d S_g }- 2\frac{\int_{\partial  B_r} \eta u \partial_r u d S_g}{\int_{\partial  B_r} \eta u^2   d S_g } \\
	&\geq {O(L)}.
	\end{align*}
%\textcolor{red}{Answer: $\int_{\partial  B_r} \eta ( \partial_r u)^2 d S \int_{\partial  B_r} \eta u^2   d S\geq (\int_{\partial  B_r} \eta u \partial_r u  d S)^2$. The subtraction of last two fractions is nonnegative. }
	This proves (\ref{need}) and thus the lemma.	
\end{proof}

Next we derive the doubling inequality with an explicit dependence on $L$.
\begin{lemma}\label{lem.Doubling.small}
Let $u$ be a solution of (\ref{newelli}) satisfying (\ref{lip2}) and (\ref{assdoub}). For a fixed constant $0<\theta\leq \frac{1}{2}$, we have
\begin{align}\label{est.smallscale}
	\|u\|_{L^2(B_{r})} \leq M^{C_1e^{{C_2L}}} \|u\|_{L^2( B_{\theta r})}
\end{align}
for $0<r\leq \frac{1}{2}$, where $C_1$ depends on $\theta$, and $C_2$ depends on $d$, $\Lambda$.
\label{classical}
\end{lemma}

%\textcolor{red}{$C_1$ and $C_2$ have different dependence on the parameters. You can choose to distinguish them if you would like to }
\begin{proof}
From (\ref{HHH}) and the definition  of $\overline{\mathcal{N}}(r)$, we have
\begin{align}
\bigg(\ln \frac{H(r)}{r^{d-1}}\bigg)'={O(L)}+ \frac{2}{r}\overline{\mathcal{N}}(r)\exp(-{CL}r).
\label{pre}
\end{align}
Note that here $O(L)$ is a function in $r$ satisfying $-CL \le O(L) \le CL$.
We would like to obtain an upper bound and a lower bound for the quotient $H(r_2)/H(r_1)$ with $0<r_1<r_2$. To find the upper bound, we integrate the equality (\ref{pre}) from $r_1$ to $r_2$ and use the monotonicity of $\overline{\mathcal{N}}(r)$ to obtain
\begin{align}
\ln \frac{H(r_2)}{r^{d-1}_2}-\ln \frac{H(r_1)}{r^{d-1}_1}\leq {CL(r_2-r_1)} + 2 \overline{\mathcal{N}}(r_2) \ln \Big( \frac{r_2}{r_1}\Big)\exp(-CLr_1).
\end{align}
Taking the exponential of both sides gives the upper bound
\begin{align}
\frac{H(r_2)}{H(r_1)}\leq e^{{CL(r_2-r_1)}} \Big(\frac{r_2}{r_1}\Big)^{2 \overline{\mathcal{N}}(r_2)\exp(-CLr_1)+d-1}.
\label{onehand}
\end{align}
To see the lower bound, we integrate (\ref{pre}) from  $r_1$ to $r_2$ and apply the monotonicity of $\overline{\mathcal{N}}(r)$  again to obtain
\begin{align}
\ln \frac{H(r_2)}{r^{d-1}_2}-\ln \frac{H(r_1)}{r^{d-1}_1}\geq -{CL}(r_2-r_1)+ 2 \overline{\mathcal{N}}(r_1)\exp(-{CLr_2}) \ln \Big( \frac{r_2}{r_1} \Big).
\end{align}
Raising to the exponential form, we have
\begin{align}
\frac{H(r_2)}{H(r_1)}\geq e^{{-CL(r_2-r_1)}} \Big(\frac{r_2}{r_1}\Big)^{2 \overline{\mathcal{N}}(r_1)\exp(-{CLr_2}) +d-1}.
\label{another}
\end{align}
Combining (\ref{onehand}) and (\ref{another}), we arrive at
\begin{align}
e^{{-CL(r_2-r_1)}} \Big(\frac{r_2}{r_1}\Big)^{2 \overline{\mathcal{N}}(r_1)\exp(-{CLr_2}) +d-1}\leq \frac{H(r_2)}{H(r_1)}\leq e^{{CL(r_2-r_1)}} \Big(\frac{r_2}{r_1}\Big)^{2 \overline{\mathcal{N}}(r_2)\exp(-CLr_1)+d-1}.
\label{largemono}
\end{align}

Next we want to show an upper bound for $\overline{\mathcal{N}}(\frac{3}{4})$. Let $r_2=\frac{3}{4}$ and $0<r_1=r<\frac{3}{4}$.   From the estimate (\ref{another}), we have
\begin{align}
e^{{-CL(\frac{3}{4}-r)}} \bigg(\frac{\frac{3}{4}}{r} \bigg)^{d-1}\leq \frac{H(\frac{3}{4})}{H(r)}.
\end{align}
Using the fact that $0<\theta\leq \frac{1}{2}$, we have
\begin{align}
\|u\|^2_{L^\infty ( B_\theta)}&\leq C \int_{ B_{\frac{3}{4}}} u^2 d V_g\leq C\int^\frac{3}{4}_0 H(r) dr \nonumber \\
&\leq C\int^\frac{3}{4}_0 r^{d-1}  e^{{CL(\frac{3}{4}-r)}} H\Big(\frac{3}{4} \Big)dr \nonumber\\
&\leq C e^{{CL}}   H\Big(\frac{3}{4}\Big),
\label{addnew}
\end{align}
where $C$ depends on $d$ and $\Lambda$.
Obviously,
\begin{align}
\|u\|^2_{L^\infty ( B_1)}&\geq C \int_{\mathbb \partial B_{{1}}} u^2 dS_g.
\label{remind}
\end{align}
Therefore, from (\ref{assdoub}), (\ref{another}) and (\ref{addnew}), we have
\begin{align}
M^2&\geq \frac {\|u\|^2_{L^\infty (B_1)}}{\|u\|^2_{L^\infty ( B_\theta)}} \geq \frac{ CH(1)}{ C  e^{{C}L} H(\frac{3}{4})}\nonumber\\
& \geq e^{{-CL}} \Big(\frac{4}{3}\Big)^{d-1+2 \overline{\mathcal{N}}(\frac{3}{4}  ) e^{{-CL}}}.
\end{align}
Thus, we can get an upper bound for $\overline{\mathcal{N}}(\frac{3}{4})$ as
\begin{align}
\overline{\mathcal{N}}\Big(\frac{3}{4}\Big)\leq  Ce^{{CL}}\ln M,
\end{align}
where $M>1$ is a large constant.
Choosing any $r\leq \frac{1}{2}$, we integrate (\ref{pre}) from $\theta r$ to $r$, by the monotonicity of $\overline{\mathcal{N}}$, we derive that
\begin{align}
\ln \frac{H(r)}{r^{d-1}}-\ln \frac{H(\theta r)}{(\theta r)^{d-1}}&\leq {CL r}+  2\overline{\mathcal{N}}(\frac{3}{4})\ln \frac{1}{\theta}  \nonumber \\
&\leq {CLr} +  C e^{{CL}} \ln M \ln \frac{1}{\theta} .
\end{align}
Thus, we obtain that
\begin{align}
H(r)&\leq  \exp( {CLr} + e^{{CL}} \ln M\ln \frac{1}{\theta} ) H(\theta r)\nonumber \\
& \leq \theta^{1-d} M^{- (\ln \theta) e^{{CL}}}  H(\theta r),
\end{align}
where $M>1$ is large. By further integrations,
we can also obtain that
\begin{align}\label{est.doub}
\|u\|_{L^2( B_{r})} \leq \theta^{\frac{-d}{2}}M^{-(\ln \theta) e^{{CL}}}\|u\|_{L^2( B_{\theta r})}
\end{align}
for $0<r\leq \frac{1}{2}$, where $C$ depends only on $d$ and $\Lambda$.
\end{proof}

\begin{remark}\label{rmk.2d}
For the case $d=2$, we introduce a new variable to apply a lifting argument. Let $v(x, t)= e^t u(x)$. Then the new function $v(x, t)$ satisfies the equation
\begin{align}
-\nabla \cdot \big(\widetilde{A}(x, t)\nabla v\big)+v=0\quad \quad \mbox{in} \ \hat{B}_1,
\end{align}
where
\begin{equation*}
\widetilde{A}(x, t)=\begin{pmatrix} A(x) & 0 \\
0 & 1   \end{pmatrix}
\end{equation*}
and $\hat{B}_1$ is the ball with radius $1$ in $\mathbb R^3$. It is easy to see that $\widetilde{A}$ satisfies the conditions (\ref{ellip}) and (\ref{lip2}). Following the procedure performed as $d\geq 3$, we are able to introduce the new metric $g$ and geodesic polar coordinates. Thus, in the metric $g$ as (\ref{metricha}) and $\eta$ as (\ref{eta}), we have
\begin{align}
-\nabla_g \cdot (\eta(x) \nabla_g v)+c_g v=0  \quad  \mbox{in} \  \hat{B}_1,
\end{align}
where $c_g=\frac{1}{\sqrt{\det{g}}}$.
As before, we could make use of the monotonicity of the frequency function to obtain the doubling inequality. Precisely, we may define
\begin{align}
D(r)=\int_{ \hat{B}_r} \eta |\nabla_g v|^2 +c_g v^2d V_g
\end{align}
and
\begin{align}
H(r)=\int_{ \partial \hat{B}_r} \eta  v^2 dS_g.
\end{align}
Then  the frequency function is defined as
\begin{align}
\mathcal{N}(r)=\frac{ rD(r)}{H(r)}.
\end{align}
Following the proof of Lemma \ref{frequency} and \cite[Theorem 3.2.1]{HL13}, we can obtain the almost monotonicity of $\mathcal{N}(r)$. That is, for any $r_0\in (0, 1)$, it holds that
\begin{align}
\exp{(CLr)}\mathcal{N}(r)\leq \exp{(CLr_0)}+ \exp{(CLr_0)}\mathcal{N}(r_0)
\end{align}
for any $r\in (0, \ r_0)$  where $C$ depends on $\Lambda$. By mimicking the argument in the proof of Lemma \ref{classical}, we can obtain the doubling inequality for $v$ in $\hat{B}_r$. This also leads to the doubling inequality for $u$ as (\ref{est.smallscale}) in $B_r$.

\begin{remark}
	For a better estimate when $d=2$, see Remark \ref{rmk.Improve2D}.
\end{remark}
%\begin{remark}
%	Theorem \ref{thm.Br} and Lemma \ref{lem.Doubling.small} combined imply a special case of Theorem \ref{thm.main} with $(\beta,\theta) = (5,\Lambda/2)$.
%\end{remark}

\end{remark}

\section{Proof of Theorem \ref{thm.main}}

This section is devoted to the proof of Theorem \ref{thm.main}.
Step 1 and Step 3 of the proof have been handled in Section 2 and Section 3, respectively. For convenience of presentation, we choose $\beta$ such that $\frac{1}{\beta-\frac{3}{4}}=5$. Our argument works for any $\beta\in (\frac{3}{4}, 1)$. To handle the case $N^{-5} \lesssim \e/r \lesssim N^{-\frac{1}{2}\tau}$ in Step 2, we will use particular doubling properties, obtained from some sort of ``analyticity'', as a transition in order to improve our estimates. Indeed, for $d\ge 3$, we will employ the three-ball inequality with a sharp exponential error term obtained in \cite{AKS20}; for $d=2$, we will use quasi-regular mappings \cite{AL08} which provides a much better doubling estimate.

We first introduce a three-ball inequality for all dimensions $d\ge 2$. For convenience, we define the normalized $L^2$ norm by
\begin{equation*}
\norm{u}_{\underline{L}^2(B_t)} = \bigg( \fint_{B_{t}} u^2 \bigg)^{1/2}.
\end{equation*}
The following theorem is essentially taken from \cite[Theorem 1.4]{AKS20}, which is a corollary of the ``large-scale analyticity'' in periodic homogenization. This result relies on the periodic structure of the coefficients, but does not depend on the smoothness of coefficients.

\begin{theorem}\label{thm.AKS}
	For each $\hat{\tau}\in (0,1/2)$, there exist $c = c(d,\Lambda)>0$ and $\theta = \theta(\hat{\tau},d,\Lambda)\in (0,1/2]$ such that if $u$ is a weak solution of $\cL_1 (u) = 0$ in $B_R$ with $\theta^2 R > 2$, then
	\begin{equation}\label{est.3ball}
	\norm{u}_{\underline{L}^2(B_{\theta R})} \le \norm{u}_{\underline{L}^2(B_{\theta^2 R})}^{\hat{\tau}} \norm{u}_{\underline{L}^2(B_{ R})}^{1-\hat{\tau}} + \exp(-c\theta^2 R) \norm{u}_{\underline{L}^2(B_{R})}.
	\end{equation}
\end{theorem}

%This lemma is precisely (4.16) of [?]. The key point is that the constants involved are independent of $m$.
%\textcolor{red}{I choose $\tau$ rather than $\sigma$, because $\sigma$ is already used in the next section. I have no better notation to use.} \textcolor{blue}{This is fine. Thanks!}

As a simple corollary, we have

\begin{corollary}\label{coro.AKS}
	Let $u$ be a weak solution of $\cL_1 (u) =0 $ in $B_R$. For every $\alpha_1>0$, there exist $C>0$ and $\theta\in (0,1/2)$ such that if
	\begin{equation}\label{est.NandR}
	N> C \quad \text{ and } \quad \theta^2 R \ge C\ln N,
	\end{equation}
	and
	\begin{equation}
	\norm{u}_{\underline{L}^2(B_R)} \le N \norm{u}_{\underline{L}^2(B_{\theta R})},
\label{NNincrease}
	\end{equation}
	then
	\begin{equation}
	\norm{u}_{\underline{L}^2(B_{\theta R})} \le C N^{1+\alpha_1} \norm{u}_{\underline{L}^2(B_{\theta^2 R})}.
\label{NNincrease2}
	\end{equation}
\end{corollary}

	The sharp exponential tail in (\ref{est.3ball}) is crucial for our purpose which is related to the condition (\ref{est.NandR}). The lower bound $\ln N$ in (\ref{est.NandR}) allows us to iterate the estimate down to a scale at which the classical theory in Section 3 may apply.

Next, for the case $d = 2$, we introduce a stronger doubling property using quasi-regular mappings (related to complex analyticity). We briefly give some background on quasi-regular mappings. For a detailed account of this topic, please refer to the presentation in \cite{AL08}, \cite[Chapter II.6]{BJS64} and references therein.

Let $u$ be a weak solution of the equation $\cL_1 (u) =0 $ in $B_R$ with only bounded measurable coefficients satisfying (\ref{ellip}).  Let $z=x+iy$ for $x, y\in\mathbb R$. Define
\begin{align}
	\partial_{\bar z} f=\frac{1}{2}(\partial_x f+i \partial_y f), \quad \partial_{z} f=\frac{1}{2}(\partial_x f-i \partial_y f).
\end{align}
We introduce a stream function (the generalized harmonic conjugate) associated with $u$ as
\begin{align*}
	\nabla v=JA\nabla u,
\end{align*}
where $J$ is the rotation matrix in the plane
\begin{align*}
	J=
	\begin{pmatrix}
		0 & -1 \medskip \\
		1& 0
	\end{pmatrix}.
\end{align*}
Let $f=u+iv$. Then we have $f\in H^1_{\text{loc}}(B_R)$ and satisfies
\begin{align*}
	\partial_{\bar z}f=\mu \partial_{ z}f+ \nu \overline{\partial_{ z}f},
\end{align*}
where the complex valued function $\mu$ and $\nu$ can be explicitly written in term of $A$ and
\begin{align*}
	|\mu|+|\nu|\leq \frac{1-\Lambda}{1+\Lambda}<1.
\end{align*}
Hence, $f: B_R\to \mathbb C$ is a $\frac{1}{\Lambda}$-quasi-regular mapping. Moreover, it can be written as $f=F\circ \hat{\chi}$, where $F$ is holomorphic and $\hat{\chi}: B_R\to B_R$ is a $\frac{1}{\Lambda}$-quasiconformal homeomorphism satisfying $\hat{\chi}(0) = 0$ and $\hat{\chi}(1) = 1$. Define
\begin{align*}
	\mathcal{B}_{\hat{r}}=\{z\in B_R: |\hat{\chi}(z)|<{\hat{r}}\}.
 \end{align*}
The quasi-balls $\mathcal{B}_{\hat{r}}$ are comparable to the standard Euclidean balls in the sense
\begin{align}
	\mathcal{B}_R=B_R, \quad \text{and} \quad B_{R(\frac{{\hat{r}}}{CR})^{\frac{1}{\alpha}}}\subset \mathcal{B}_{\hat{r}}\subset B_{R(\frac{C{\hat{r}}}{R})^{{\alpha}}}, \quad \mbox{for} \ r<R,
	\label{quasi-rel}
\end{align}
where $C\geq 1$ and $0<\alpha<1$ depend only on $\Lambda$. Observe that $\mathcal{B}_{\hat{r}}$ tends to be singular if $\hat{r}\ll R$, which fortunately is not too restrictive as we only use it in the transition at intermediate scales.

From the fact that $F$ is a holomorphic function, the following doubling property  holds \cite{AL08}.
\begin{lemma}\label{lem.confor}
	If $u\in H^1_{{\rm loc}}(B_R)$ is a nonzero weak solution of $\cL_1 (u) =0 $ in $B_R$, then
	\begin{align}
		\frac{\|u\|_{L^\infty(\mathcal{B}_{\hat{r}})}} {\|u\|_{L^\infty(\mathcal{B}_\frac{{\hat{r}}}{2})}} \leq C \frac{\|u\|_{L^\infty(\mathcal{B_R})}} {\|u\|_{L^\infty(\mathcal{B}_\frac{R}{4})}}, \quad \mbox{for} \ 0< {\hat{r}}\leq R.
		\label{dou-confor}
	\end{align}
\end{lemma}

\begin{remark}
	Note that Lemma \ref{lem.confor} does not use periodicity, and it is also true for solutions of $\cL_{\e}(u_\e) = 0$, with a constant independent of $\e$. This gives ``almost monotonicity of the doubling constant'', a statement stronger than that of Theorem \ref{thm.main}, but for quasi-balls, as opposed to the usual balls. As we pointed out above, quasi-balls are difficult to manage as the radius $r$ goes to zero. Because of this, we still need to use the periodicity assumption and Step 1 below, when $d = 2$, to cover the range $r\ge CN^5 \e$. We then apply Lemma \ref{lem.confor} in the range $C\e<r<CN^5 \e$, using (\ref{quasi-rel}) since $r$ is not too small. Finally, the case $0<r<C\e$ is handled by scaling and Lemma \ref{lem.Doubling.small}. The details are below.
\end{remark}

Equipped with Corollary \ref{coro.AKS} and Lemma \ref{lem.confor}, we are ready to prove Theorem \ref{thm.main}.

\begin{proof}[Proof of Theorem \ref{thm.main}]

According to the relationship between $\e$ and $N$, one needs to consider three cases based on the comparison of $\e$ with $ N^{-5}$ and $N^{-\frac{\tau}{2}}$ (or $1$ for $d= 2$). Without loss of generality, we may just consider the most complicated case $\e \lesssim N^{-5}$, since all the three steps listed in the introduction will be involved as $r$ approaches 0. Hence, we fix $\e$ and $N$ so that $C N^5 \e < 1$, and then discuss the different ranges of $r$.
	
	\textbf{Step 1:} $CN^5\e < r< 1$. Under either (\ref{est.B1}) or (\ref{est.B1-2}), Theorem \ref{thm.Br} implies
	\begin{equation}\label{est.step1}
	\int_{B_{r}} u_\e^2 \le 8N^3 \int_{B_{\theta r}} u_\e^2,
	\end{equation}
	for any given $\theta \in (0,\frac{\Lambda}{2}]$. This estimate holds for all dimensions $d\ge 2$.
	
	\textbf{Step 2:} In this step, we need to treat the cases $d\ge 3$ and $d=2$ separately.
	
	\textit{Case 1:} $d\ge 3$ and $C\e N^{\frac{\tau}{2}} < r < CN^5\e$ for any fixed $\tau>0$. Let $m$ be the smallest integer so that $\theta^{-m} r > CN^5\e$. If $N$ is bounded by some absolute constant, then Step 2 is not needed. Since $r>C\e N^{\frac{\tau}{2}}$,  for sufficiently large $N$, $m$ satisfies
	\begin{equation}\label{est.mbound}
	m \le \frac{6\ln N}{-\ln \theta}.
	\end{equation}
	Because of (\ref{est.step1}), we have
	\begin{equation}
	\int_{B_{\theta^{-m} r}} u_\e^2 \le 8 N^3 \int_{B_{\theta^{-m+1} r}} u_\e^2.
	\end{equation}
	Let $M_0 = 8 N^3 $ and  $M_j$ be the constant such that
	\begin{equation}
	\int_{B_{\theta^{-m+j} r}} u_\e^2 \le M_j \int_{B_{\theta^{-m+j+1} r}} u_\e^2.
	\end{equation}
	The goal is to estimate $M_m$ with $m$ comparable to the bound in (\ref{est.mbound}).
	
	Thanks to Corollary \ref{coro.AKS}, and by rescaling, we know that for a given $\alpha_1>0$ with $\theta$ small enough, we have
	\begin{equation}
	M_j = CM_{j-1}^{1+\alpha_1}.
	\end{equation}
	Note that the left-end restriction $r>C\e \ln M_j$ is needed in order to apply Corollary \ref{coro.AKS}, due to (\ref{est.NandR}). This can be guaranteed if we eventually show $M_j \le M_m < C\exp(N^{\frac{\tau}{2}})$.
	
	We now proceed to estimate $M_j$. Using the initial condition $M_0 = 8N^3$, one can show explicitly that
	\begin{equation}
	M_j = \exp(-\ln C/\alpha_1) \exp \Big[ (1+\alpha_1)^j \big(3\ln N + \ln(8C^{1/\alpha_1}) \big) \Big].
	\end{equation}
	It follows from (\ref{est.mbound}) that
	\begin{equation}
	M_m \le C\exp\Big[ \exp\big( \ln(1+\alpha_1)(-\ln\theta)^{-1} 6\ln N \big) \cdot  \big(3\ln N + \ln(8C^{1/\alpha_1}) \big) \Big].
	\end{equation}
	Note that $\tau$ is any given positive constant. Then, we may choose $\alpha_1$ small enough (hence $\theta$ is also small), so that
	\begin{equation}
	\frac{\tau}{3} \ge 6 \ln(1+\alpha_1)(-\ln\theta)^{-1}.
	\end{equation}
	Thus, if $N$ is large enough,
	\begin{equation}
	M_m \le C\exp(N^{\frac{1}{2}\tau}).
	\end{equation}
	This implies that for any $CN^{\frac{1}{2}\tau} \e < r < CN^5\e$, we have
	\begin{equation}\label{keyest}
	\int_{B_{r}} u_\e^2 \le C\exp(N^{\frac{1}{2}\tau}) \int_{B_{\theta r}} u_\e^2.
	\end{equation}

	\textit{Case 2:} $d = 2$ and $C\e < r < CN^5 \e$. From (\ref{est.step1}) with $\theta = \frac{\Lambda}{2}$ in Step 1, for $R \simeq C\e N^{5}$,
	\begin{align*}
		\int_{B_{R}} u_\e^2 \le 8N^3 \int_{B_{\frac{\Lambda}{2}R}} u_\e^2.
	\end{align*}
	By the $L^\infty$ norm estimates, it  follows that
	\begin{align}
		\|u_\e\|_{L^\infty(B_R)} \le CN^{\frac{3}{2}} \|u_\e\|_{L^\infty(B_\frac{R}{2})}.
		\label{ite-con}
	\end{align}
	We would like to to apply Lemma \ref{lem.confor} to $u_\e$. From the relation (\ref{quasi-rel}) of quasi-balls $\mathcal{B}_{\hat{r}}$ and the standard balls, as well as the iteration of the doubling inequality (\ref{ite-con}), we have
	\begin{align}\label{est.Quasiball-Ball}
		\frac{\|u_\e\|_{L^\infty(\mathcal{B}_R)}} {\|u_\e \|_{L^\infty(\mathcal{B}_\frac{R}{4})}} \leq
		\frac{\|u_\e\|_{L^\infty({B_R})}} {\|u_\e \|_{L^\infty({B_{R(4C)^{-\frac{1}{\alpha}}}})}} \leq CN^k,
	\end{align}
	where $k$ depends only on $\Lambda$.
	Thus, (\ref{dou-confor}) implies that for any $0<\hat{r}<R$
	\begin{align}
		\frac{\|u_\e\|_{L^\infty(\mathcal{B}_{\hat{r}})}} {\|u_\e\|_{L^\infty(\mathcal{B}_\frac{{\hat{r}}}{2})}} \leq CN^k.
	\end{align}
	In order to establish a doubling inequality at small scale on standard Euclidean balls, we iterate the above doubling inequality $m$ times to obtain
	\begin{align}
		{\|u_\e\|_{L^\infty(\mathcal{B}_{\hat{r}})}}  \leq (CN^k)^m {\|u_\e\|_{L^\infty(\mathcal{B}_\frac{{\hat{r}}}{2^m})}}.
	\end{align}
	%We want to have the doubling inequality in the standard balls.
	By the relation (\ref{quasi-rel}),
	\begin{align*}
		\|u_\e\|_{L^\infty(B_{R(\frac{{\hat{r}}}{CR})^{\frac{1}{\alpha}}})} \leq ( CN^{k})^m  \|u_\e\|_{L^\infty(B_{R(\frac{C{\hat{r}}}{2^mR})^{{\alpha}}})}.
	\end{align*}
	We choose $m>0$ to be the smallest integer so that
	\begin{equation}
		R(\frac{C{\hat{r}}}{2^mR})^{{\alpha}} \leq \frac{1}{2}R(\frac{{\hat{r}}}{CR})^{\frac{1}{\alpha}}.
		\label{assum-m}
	\end{equation}
	Consequently,
	\begin{equation}\label{dou-mm}
		\|u_\e\|_{L^\infty(B_{R(\frac{{\hat{r}}}{CR})^{\frac{1}{\alpha}}})} \leq ( CN^{k})^m \|u_\e\|_{L^\infty(B_{\frac{1}{2}R(\frac{{\hat{r}}}{CR})^{\frac{1}{\alpha}}})}.
	\end{equation}
	Note that $\hat{r}$, satisfying $0< \hat{r} < R \simeq C\e N^5$, is arbitrary and $m$ is chosen depending on $\hat{r}$. We now assume ${\hat{r}} \ge C \e^{\alpha}  R^{1-\alpha} \simeq C\e N^{5(1-\alpha)}$.
	Hence, $r: = R(\frac{{\hat{r}}}{CR})^{\frac{1}{\alpha}} \ge C\e$. Moreover, from (\ref{assum-m}), we have $m \le C\ln N$, where $C$ depends only on $\Lambda$. Thus, it follows from (\ref{dou-mm}) that
	\begin{align}\label{est.2D.>Ce}
		\|u_\e\|_{L^\infty(B_{r})} &\leq CN^{C\ln N} \|u_\e\|_{L^\infty(B_{\frac{r}{2}})}
	\end{align}
	for all $C\e < r< R \simeq C\e N^5$. Since  $L^\infty$ norm can be replaced by $L^2$ norm in the above inequality, we derive the desired estimate for the case $d = 2$.

	\textbf{Step 3:} For $r< C\e N^{\frac{1}{2}\tau}$ (or $r<C\e$ for $d = 2$), by rescaling,
	the equation may be reduced to the case in which the Lipschitz constant of coefficients is bounded by $CN^{\frac{\tau}{2}}$ (bounded by $C$ for $d=2$). It follows from (\ref{est.smallscale}) and (\ref{keyest}) that
	for $d\ge 3$ and any $0< r< C\e N^{\frac{1}{2}\tau}$,
	\begin{equation*}
	\begin{aligned}
	\int_{B_{r}} u_\e^2 & \le C\Big[ \exp(N^{\frac{1}{2}\tau}) \Big]^{C \exp(N^{\frac{1}{2}\tau})} \int_{B_{\theta r}} u_\e^2 \\
	& \le \exp(\exp(CN^\tau))\int_{B_{\theta r}} u_\e^2.
	\end{aligned}
	\end{equation*}
	For $d = 2$, it follows from (\ref{est.smallscale}) and (\ref{est.2D.>Ce}) that for any $0< r<C\e$,
	\begin{equation}\label{key}
		\int_{B_{r}} u_\e^2 \le C \big[ N^{C\ln N} \big]^C \int_{B_{\frac{r}{2}}} u_\e^2 \le CN^{C\ln N} \int_{B_{\frac{r}{2}}} u_\e^2.
	\end{equation}
	Note that $N^{C\ln N} = \exp(C(\ln N)^2)$. This completes the proof of Theorem \ref{thm.main}.
%		By rescaling, we can apply the case in which the Lipschitz constant of the coefficients is $L=C$. It follows from (\ref{est.smallscale}) that
%	\begin{align}
%		\|u_\e\|_{L^\infty(B_{r})} &\leq  e^{C(\ln N)^2} \|u_\e\|_{L^\infty(B_{\frac{1}{2}r})}
%		\label{quasi-r1}
%	\end{align}
%	for $0<r\leq C\e$. We are left with proof of the doubling inequality for the case $C\e\leq r_0\leq C\e N^{\frac{1}{\beta-\frac{3}{4}}}$.
%	We choose $\hat{r}$ such that $R(\frac{{\hat{r}}}{CR})^{\frac{1}{\alpha}}= r_0= C\e N^{\tau_0}$ for any $0\leq \tau_0\leq \frac{1}{\beta-\frac{3}{4}}$.
%	Then $\hat{r}=C \e^{\alpha} N^{\alpha \tau_0} R^{1-\alpha}$. To satisfy the condition (\ref{assum-m}), we can still choose $m=C\ln N$. From (\ref{dou-mm}), we arrive at the doubling inequality
%	\begin{align}
%		\|u_\e\|_{L^\infty(B_{r})} &\leq  e^{C(\ln N)^2} \|u_\e\|_{L^\infty(B_{\frac{1}{2}r})}
%		\label{quasi-r2}
%	\end{align}
%	for $ C\e\leq r\leq C\e N^{\frac{1}{\beta-\frac{3}{4}}}$. Therefore, the combination of Theorem \ref{thm.Br}, (\ref{quasi-r1}) and (\ref{quasi-r2}) yields the estimate (\ref{est.Br-2}).
%	
\end{proof}

\begin{remark}\label{re-three}
	It was shown in \cite{AKS20} that the exponential tail in (\ref{est.3ball}) is sharp (up to the end point $\hat{\tau}=\frac{1}{2}$), without any smoothness assumption on the coefficients. If the critical $\hat{\tau} = 1/2$ in (\ref{est.3ball}) can also be achieved (which seems like a very difficult task), then Corollary \ref{coro.AKS} with $\alpha_1 = 0$ would follow. By the argument in Step 2, this would yield the estimate
\begin{align}\label{est.1/2IsTrue}
\int_{B_r} u_\e^2\leq CN^k \int_{B_{\theta r}} u_\e^2
\end{align}
for $C\e\ln N\leq r\leq C \e N^5$. If we then apply (\ref{est.smallscale}) as in Step 3, with Lipschitz constant $C \ln N$, we would obtain the bound $C(N)=\exp{(CN^C)}$ for $0<r\leq \frac{1}{2}$ (for the range $0<r<C\e$, (\ref{est.smallscale}) does give the optimal bound). On the other hand, the estimate (\ref{est.smallscale}) in term of the large Lipschitz constant $L$ may not be sharp. This is a well-known difficult issue in quantitative unique continuation, for which none of the currently known methods apply. Any improvement here would have many consequences. Alternatively, in the range $C \e\leq r\leq C \e \ln N$, one could try to use a method taking advantage of both periodicity and smoothness. No such method is available at the moment.

For $d=2$, if the critical case $\hat{\tau} = \frac{1}{2}$ in (\ref{est.3ball}) is true, then we can actually show our expected estimate which has very important consequences for the study of long-standing open problems regarding the spectral properties of second order elliptic operators with periodic coefficients.
Note though that when $d=2$, these problems have already been solved in \cite{M98,M00} (also see \cite[subsection 7.3, 7.4]{KL02}). However, the approach just outlined would require lower regularity on the coefficients than \cite{M98,M00,KL02}.
To obtain the expected estimate assuming this critical case holds, note that (\ref{est.1/2IsTrue}) holds for $r>C\e\ln N$. By reproducing the argument in Step 2 (Case 2), we can then show $m \le C\ln\ln N$ and therefore
\begin{equation}\label{est.ultimate}
	\|u_\e\|_{L^\infty(B_{r})} \leq CN^{C\ln \ln N} \|u_\e\|_{L^\infty(B_{\frac{r}{2}})}
\end{equation}
for all $r>C\e$. Then a blow-up argument gives the same estimate for $0<r<C\e$. Observe that (\ref{est.ultimate}) is exactly the ultimate estimate we expect, as mentioned in the introduction.
	%\textcolor{red}{ Corollary \ref{coro.AKS} basically states that the doubling inequality in a large ball implies the doubling inequality in a relatively small ball. Most likely, the power of $N$ in  (\ref{NNincrease}) will increase as (\ref{NNincrease2}) in the doubling inequality or $C$ depends on the leading coefficient $A(x)$ of the operator $\cL_1$ in (\ref{NNincrease2}). In the Step 2 of the above proof, we do not use the Lipschitz continuity of $A(x)$, which becomes very large after rescaling to the operator $\cL_1$. Thus, $\alpha_1=0$ seems unlikely to hold without further restriction on $C$ in (\ref{NNincrease2}).}
\end{remark}

\begin{remark}\label{rmk.Improve2D}
	If we consider, when $d=2$, elliptic operators with Lipschitz coefficients, with Lipschitz constant $L>1$ (and no periodicity assumption), we can obtain the improved bound $M^{C_1 \ln L}$ in Lemma \ref{lem.Doubling.small}. To show this, we break down the scales into $\frac{1}{L} \le r < 1$ and $0<r<\frac{1}{L}$. For the case $\frac{1}{L} \le r<1$, we use Lemma \ref{lem.confor}, (\ref{quasi-rel}) and the argument from (\ref{est.Quasiball-Ball}) to (\ref{dou-mm}). For the case $0<r<\frac{1}{L}$, we scale to reduce to the case $L=1$ and then apply Lemma \ref{lem.Doubling.small} as it stands. This may suggest that the bound in Lemma \ref{lem.Doubling.small} is not optimal, also for $d\ge 3$. 
\end{remark}

%\begin{remark}\label{resharp}
%\textcolor{red}{The double exponential growth $\exp(\exp(CN^\tau))$ in  (\ref{est.Br}) in Theorem \ref{thm.main} seems to be almost sharp at the moment. Because of the Remark \ref{re-three}, we can at most reduce the doubling inequality to the small scales $r\leq \e N^\tau$ for any $\tau>0$. In this small scale, the coefficient in the doubling inequality should be double exponential growth because of the large Lipschitz constant.}
%\end{remark}

\begin{remark}\label{rmk.largetheta}
	The disadvantage of Theorem \ref{thm.AKS} for $d\ge 3$ is that $\theta$ may be very small. If we do not apply Theorem \ref{thm.AKS} to improve the exponent $\tau$, Step 1 and 3 in the proof of Theorem \ref{thm.main} allows $\theta$ to be any number in $(0, \Lambda/2]$. In particular, under (\ref{precond}), for any $\beta\in (\frac{3}{4}, 1)$, we have
	\begin{equation}
	\int_{B_r} u_\e^2 \le \exp (\exp (CN^{\frac{1}{\beta-\frac{3}{4}}}))\int_{ B_{\Lambda r/2}} u_\e^2.
\label{doudouble.es}
	\end{equation}
	For convenience, we will use this doubling inequality (\ref{doudouble.es}), instead of (\ref{est.Br}), in estimating the upper bound of nodal sets in the next section. The price is that $\alpha$ has to be larger than $8$ in (\ref{est.Nodal}).

%If we choose $\theta=\frac{1}{2}$, it follows from Lemma \ref{lem.Ak}
% that the following doubling inequality holds
%	\begin{equation}
%	\int_{E_{2r}} u_\e^2 \le 2N \int_{E_{ r}} u_\e^2
%\label{double-le}
%	\end{equation}
%for $C\e N^\frac{1}{\beta-\frac{3}{4}} \le r \le \frac{1}{2}$. %\textcolor{red}{Explain why we can replace $E_r$ by $B_r$}
\end{remark}

\section{Upper bounds  of Nodal sets}
In this section, we study of the upper bounds of nodal sets for $u_\e$, where $u_\e$ is a nonzero solution of $\mathcal{L}_\e (u_\e)=0$ satisfying (\ref{precond}). We will focus on the general treatment for all dimensions $d\ge 2$ and with an eye towards $d=2$ in the end.
Throughout this section, up to a change of variable, we assume $\mathcal{L}_0=-\Delta$. Note that in this case, $E_r$'s are just balls, and in view of Theorem \ref{thm.large.r}, the assumption (\ref{precond}) can be replaced by
\begin{equation}\label{NewPrecond}
\int_{B_2} u_\e^2 \le N \int_{B_1} u_\e^2,
\end{equation}
 and (\ref{doudouble.es}) holds with $\Lambda = 1$.

%	
%	 \textcolor[rgb]{0.00,0.50,1.00}{We adopt the estimate (\ref{LSS}) instead of (\ref{est.ue-u0}) for the approximation of $u_\e$, because the bound $\e$ in (\ref{LSS}) provides a stronger approximation for $u_\e$. Note that $u_0$ in (\ref{est.ue-u0}) is the homogenized solution and $u_0$ in (\ref{LSS}) is the approximation of $u_\e$ satisfying $\mathcal{L}_0(u_0)=0$ in the interior domain.} By a rescaling, we assume that $\mathcal{L}_0(u_0)=\Delta u_0=0$.
%	\textcolor{red}{We need a remark here to explain the difference between the $u_0$ in (\ref{lem.1step}) and the $u_0$ in the last Lemma. They are both critical with their own advantages.}
\subsection{Small scales} We first show that a doubling inequality centered at $0$ implies the doubling inequality with shifted centers.
\begin{lemma}\label{lem.doublingBx}
	Let $u_\e$ be a weak solution of $\cL_\e (u_\e) = 0$ in $B_2$ satisfying (\ref{NewPrecond}). Then for any $x\in B_{1/3}$ and $B_{2r}(x) \subset B_2$, we have
	\begin{equation}\label{generaldouble}
	\int_{B_{2r}(x) } u_\e^2 \le \exp(\exp(CN^{\frac{2}{\beta-\frac{3}{4}}})) \int_{B_r(x)} u_\e^2.
	\end{equation}
\end{lemma}
\begin{proof}
	Let us first assume $C\e N^{\frac{1}{\beta-\frac34}} <1$ for some large $C$. In this case, by Theorem \ref{thm.large.r} with $\theta = 1/2$, we have
	\begin{equation}\label{est.B2-B.25}
	\int_{B_2} u_\e^2 \le N \int_{B_1} u_\e^2 \le 2N^2 \int_{B_{1/2}} u_\e^2.
	\end{equation}
	Now, for any $x\in B_{1/3}$, note that $B_{1/2} \subset B_{5/6}(x)$ and $B_{5/3}(x) \subset B_2$. It follows from (\ref{est.B2-B.25}) that
	\begin{equation}
	\int_{B_{5/3}(x)} u_\e^2 \le \int_{B_2} u_\e^2 \le 2N^2 \int_{B_{1/2}} u_\e^2 \le 2N^2 \int_{B_{5/6}(x)} u_\e^2.
\label{comparedouble}
	\end{equation}
	Since Theorem \ref{thm.large.r} and (\ref{doudouble.es}) are invariant under translation, we can apply them in $B_{5/3}(x)$ with $N$ replaced by $2N^2$. Thus, for all $r\in (0,5/6)$,
	\begin{equation*}
	\int_{B_{2r}(x)} u_\e^2 \le \exp(\exp(CN^{\frac{2}{\beta-\frac{3}{4}}})) \int_{B_{r}(x)} u_\e^2.
	\end{equation*}
	
	To handle the case $C\e N^{\frac{1}{\beta-\frac34}} \ge 1$, we use (\ref{doudouble.es}) directly and obtain
	\begin{equation*}
	\int_{B_2} u_\e^2 \le N \int_{B_1} u_\e^2 \le \exp(\exp(CN^{\frac{1}{\beta-\frac{3}{4}}})) \int_{B_{1/2}} u_\e^2.
	\end{equation*}
	Then the desired estimate follows from the same idea as the first case and a blow up argument as in Step 3 in the proof of Theorem \ref{thm.main}.
\end{proof}

%Under (\ref{precond}), by Remark \ref{rmk.largetheta}, for any $\beta \in (\frac{3}{4}, 1)$, we have the following doubling inequality
%\begin{align}
%\int_{B_{2r}(x) } u_\e^2 \le \exp(\exp(CN^{\frac{1}{\beta-\frac{3}{4}}})) \int_{B_r(x)} u_\e^2,
%\label{generaldouble}
%\end{align}
%for any $0<r<1/2$ and $x\in B_{1/2}$.

Let us define the nodal sets as
\begin{align}
Z(u_\e)=\{x\in B_2| u_\e=0\}
\end{align}
and the density function of nodal sets as
\begin{align}
E_\e(y, r)=\frac{ H^{d-1}(Z(u_\e)\cap B_r(y))}{r^{d-1}} .
\end{align}
Based on Lemma \ref{lem.doublingBx} and a blow up argument, we can estimate the Hausdoff measure of the nodal set of $u_\e$ in small balls.

\begin{lemma}\label{smallscale}
	For any $0<r < 1/3$ and $x_0\in B_{1/3}$ such that $B_r(x_0) \subset B_{1/3}$,
	\begin{align}
	E_\e(x_0,r) \leq \Big(1+ \frac{r}{\e}\Big) \exp( CN^{\frac{2}{\beta-\frac{3}{4}}}),
	\label{newnewdouble2}
	\end{align}
	where $C$ depends on $d, \Lambda, \beta$ and $\gamma$.
\end{lemma}
\begin{proof}
	First of all, we consider the case $0<r\le \e$ and $B_r(x_0) \subset B_{1/3}$. Let $v(x)=u_\e (x_0+rx)$ and $A^{\e,r}_{x_0}(x) = A(\e^{-1}(x_0+rx))$. Then
	\begin{align}
	\nabla (A^{\e,r}_{x_0}(x) \nabla v(x))=0.
	\end{align}
	By (\ref{cond.Lip}),
	\begin{align}
	|A^{\e,r}_{x_0}(x)-A^{\e,r}_{x_0}(y)|\leq \gamma r \e^{-1} |x-y| \le \gamma |x-y|
	\end{align}
%	\begin{align}
%	|A^{\e,r}_{x_0}(x)-A^{\e,r}_{x_0}(y)|\leq \gamma r \e^{-1} |x-y| \le CN^{4+\alpha}|x-y|.
%	\end{align}
	for $x, y\in B_2$. Therefore, in this case, the coefficient matrix has a uniform Lipschitz constant independent of $\e$ and $N$.
	Then, a change of variable and the doubling inequality in Lemma \ref{lem.doublingBx} give that
	\begin{align}
	\fint_{B_{2}} v^2 \, dx&=\fint_{B_{2 r}(x_0)} u_\e^2 \, dx \nonumber \\ &\leq  \exp(\exp(CN^{\frac{2}{\beta-\frac{3}{4}}}))
	\fint_{B_{ r}(x_0)} u_\e^2 \, dx  \nonumber \\ & \leq \exp(\exp(CN^{\frac{2}{\beta-\frac{3}{4}}})) \fint_{B_{1}} v^2 \, dx.
	\end{align}

	By the upper bound of nodal sets in \cite{L18-1}, there exists a constant $\beta_0>\frac{1}{2}$ so that
	\begin{align}
	H^{d-1}(Z(v)\cap B_1)\leq  \big[ \exp  (CN^{\frac{2}{\beta-\frac{3}{4}}}) \big]^{\beta_0}\le \exp  (C_1 N^{\frac{2}{\beta-\frac{3}{4}}}),
	\end{align}
	which implies, by rescaling,
	\begin{equation*}
	H^{d-1}(Z(u_\e)\cap B_r(x_0))\leq \exp( C_1 N^{\frac{2}{\beta-\frac{3}{4}}})\; r^{d-1}
	\end{equation*}
	for any $r\in (0,\e]$ and $B_{r}(x_0) \subset B_{1/3}$.
	
	Next, to deal with the case $r>\e$, we simply use a covering argument. Let $x_0\in B_{1/3}$ and $r>\e$. There there exists a family of balls $B_\e(x_i), i = 1,2,\cdots, M,$ that covers $B_r(x_0)$ with a finite number of overlaps depending only on $d$. Note that $M \approx (r/\e)^d$. Consequently,
	\begin{equation*}
	\begin{aligned}
	H^{d-1}(Z(u_\e) \cap B_r(x_0)) & \le \sum_{i=1}^M H^{d-1}(Z(u_\e) \cap B_\e(x_i)) \\
	& \le M \exp( C_1 N^{\frac{2}{\beta-\frac{3}{4}}})\; \e^{d-1} \\
	& \le C r^d \e^{-1} \exp( C_1 N^{\frac{2}{\beta-\frac{3}{4}}}).
	\end{aligned}
	\end{equation*}
	We obtain the desired estimate by enlarging the constant $C_1$.
\end{proof}

\begin{remark}
	The above lemma does not rely on the periodicity of the coefficients. Actually, its proof also gives how the estimate depends on the Lipschitz constant of the coefficients. Precisely, if $v$ is a solution of $\nabla\cdot (A(x)\nabla v) = 0$ in $B_2$. In addition to the ellipticity condition (\ref{ellip}), we assume
	\begin{align}
	|A(x)-A(y)|\leq L |x-y|.
	\end{align}
	Then
	\begin{equation*}
	E_\e(x_0,r) \leq C (1+ Lr) N(v,Q)^{\beta_0},
	\end{equation*}
for $B_{r}(x_0) \subset Q$,	where the definition of $N(v,Q)$ is given below in (\ref{indexvv}).
%\textcolor{blue}{(The doubling index is a little restricted notion. We need to shrink the estimated domain to the smaller region. }
\end{remark}
\subsection{Large scales} To deal with the nodal sets at large scales, we need to use the homogenization theory. Precisely, in the following, we find an approximate solution $u_0$, close to $u_\e$ under $L^\infty$ norm, and satisfying a doubling inequality.
%To begin with, we establish a doubling inequality for an approximate solution $u_0$, close to $u_\e$.

\begin{lemma}\label{lem.1step}
	Suppose $r> 3C\sqrt{N} \e$ for some large $C$. Let $u_\e$ be a solution of $\cL_{\e} (u_\e) = 0$ in $B_{2r}$  satisfying
	\begin{equation}\label{cond.B2r}
	\int_{B_{2r}} u_\e^2 \le N\int_{B_{r}} u_\e^2.
	\end{equation}
	Then there exists $u_0$ satisfying $\cL_0 (u_0) = 0$ in $B_\frac{7r}{4}$ such that
	\begin{align}
	\|u_\e-u_0\|_{L^\infty( B_{\frac{3r}{2}})}\leq \frac{C\e}{r}  \|u_\e\|_{\underline{L}^2 (B_{2r})},
	\label{LSS}
	\end{align}
	and
	\begin{equation}
	\int_{B_r} u_0^2 \le 16 N^2 \int_{B_{r/2}} u_0^2
	\label{unbelieve},
	\end{equation}
	where $C$ depends on $d, \Lambda$ and $\gamma$.
\end{lemma}
\begin{proof}
	
	%From the co-area formula, there exists some $c\in (\frac{1}{2}, 2)$ so that
	% \begin{align}
	%	\int_{\partial B_{2-ct}} |\nabla u_\e|^2 + \int_{\partial B_{2-ct}}  u_\e^2 \le \frac{CN}{t^3} \int_{B_1} u_\e^2.
	%	\end{align}
	%	Without loss of generality, let us simply assume $c = 1$.
	By rescaling, we may assume $r = 1$. The construction of such locally homogenized solution $u_0$ and the estimate (\ref{LSS}) can be found in \cite[Theorem 2.3]{LS19}.
	Note that it is not necessary that $u_\e=u_0$ on $\partial B_{\frac{7}{4}}$.
	Then, it suffices to show (\ref{unbelieve}). By (\ref{cond.B2r}) and (\ref{LSS}), we have
	\begin{align}
	\|u_\e-u_0\|^2_{L^\infty( B_{\frac{3}{2}})}\leq C \e^2 \int_{ B_{2}}  u_\e^2 \le  {C\e^2N}  \int_{B_1} u_\e^2.
	\label{first.es}
	\end{align}
	We now establish estimates to compare the norms of $u_\e$ and $u_0$. Thanks to (\ref{first.es}),
	\begin{align}
	\|u_0\|_{L^2(B_{\frac{3}{2}})}&\leq \|u_\e\|_{L^2(B_{2})}+ \|u_\e-u_0\|_{L^2(B_{\frac{3}{2}})} \nonumber \\
	&\leq\sqrt{N}\|u_\e\|_{L^2(B_{1})}+{C\sqrt{N}\e}\|u_\e\|_{L^2(B_{1})} \nonumber \\
	&=\sqrt{N}(1+C\e )\|u_\e\|_{L^2(B_{1})}.
	\label{esta}
	\end{align}
	By the same strategy, using (\ref{first.es}), we obtain that
	\begin{align}
	\|u_\e\|_{L^2(B_{1})}&\leq \|u_\e-u_0\|_{L^2(B_{1})}+\|u_0\|_{L^2(B_{1})} \nonumber \\
	&\leq C\sqrt{N}\e \|u_\e\|_{L^2(B_{1})}+\|u_0\|_{L^2(B_{1})}.
	\label{stra}
	\end{align}
	Since
	$C\sqrt{N}\e \leq \frac{1}{3}$, the above estimate yields
	\begin{align}
	\|u_\e\|_{L^2(B_{1})}\leq (1+C\sqrt{N} \e) \|u_0\|_{L^2(B_{1})}.
	\label{obtain}
	\end{align}
	Combining (\ref{esta}) and (\ref{obtain}) together yields that
	\begin{align}
	\|u_0\|_{L^2(B_{\frac{3}{2}})}&\leq \sqrt{N}(1+C\sqrt{N} \e )\|u_0\|_{L^2(B_{1})} \nonumber \\
	& \leq 2 \sqrt{N}\|u_0\|_{L^2(B_{1})}.
	\label{noreal}
	\end{align}
	
	Now, we use the fact that
	\begin{equation}
	\varphi(s) = \log_2 \fint_{B_{2^s}} u_0^2
	\label{convex}
	\end{equation}
	is a convex function  with respect to $s$.  Then $f(s)=\varphi(s)-\varphi(s-c)$ is nondecreasing for any $c>0$. This implies
	\begin{align}
	\|u_0\|_{L^2(B_{1})}&\leq 2 \sqrt{N} \|u_0\|_{L^2(B_{\frac{2}{3}})}\nonumber  \\
	&\leq 4{N} \|u_0\|_{L^2(B_{\frac{4}{9}})} \nonumber  \\
	&\leq 4{N} \|u_0\|_{L^2(B_{\frac{1}{2}})}.
	\label{useful}
	\end{align}
	This proves (\ref{unbelieve}) and the lemma.
	% We aim to establish a doubling inequality for $u_\e$.
	%\begin{align}
	%\|u_\e \|_{L^2(B_{1})}&\leq \|u_\e -u_0\|_{L^2(B_{1})}+  \|u_0\|_{L^2(B_{1})} \\
	% &\leq C \sqrt{N} \e t^{\frac{-3}{2}} \|u_\e \|_{L^2(B_{1})}+ 4{N} \|u_0\|_{L^2(B_{\frac{4}{9}})}.
	%\end{align}
	% Thus, we have
	%\begin{align}
	%\|u_\e \|_{L^2(B_{1})}&\leq C{N} \|u_0\|_{L^2(B_{\frac{4}{9}})}.
	%\label{right}
	%\end{align}
	%We estimate the right hand side of (\ref{right}).
	%\begin{align}
	%\|u_0\|_{L^2(B_{\frac{4}{9}})}\leq \|u_\e-u_0\|_{L^2(B_{\frac{4}{9}})}+ \|u_0\|_{L^2(B_{\frac{1}{2}})}
	%\label{cons}
	%\end{align}
	% We estimate $\|u_\e-u_0\|_{L^2(B_{\frac{4}{9}})}$ as follows. By Lin and Shen' theorem 2.3, trace theorem and Caccipolli's inequality, we have
	%\begin{align}
	%\|u_\e-u_0\|_{L^2(B_{\frac{4}{9}})}&\leq C\e\|u_\e\|_{H^{\frac{1}{2}}(\partial B_{\frac{4+9\theta}{9}})} \\
	%&\leq C\e\|u_\e\|_{H^1( B_{\frac{4+9\theta}{9}})} \\
	%&\leq C\e\|u_\e\|_{H^1( B_{\frac{4+18\theta}{9}})}\\
	%&\leq C\e \|u_\e\|_{H^1( B_{\frac{1}{2}})}
	%\label{frac}
	%\end{align}
	%as $\theta\leq \frac{1}{4}(\frac{1}{2}-\frac{4}{9})$. Together with (\ref{cons}) and (\ref{frac}), we derive that
	%\begin{align}
	% \|u_0\|_{L^2(B_{\frac{4}{9}})}\leq C \|u_\e \|_{L^2(B_{\frac{1}{2}})}.
	%\end{align}
	% From (\ref{right}), we arrive at
	% \begin{align}
	%\|u_\e \|_{L^2(B_{1})}&\leq C{N} \|u_\e\|_{L^2(B_{\frac{1}{2}})}
	%\end{align}
	% for $\e\leq c N^{\frac{-1}{2}}$.
	% This proves the lemma.
\end{proof}

\begin{remark}
	We would like to point out that the advantage of Lemma \ref{lem.1step}, compared to (\ref{est.ue-u0}), is that it provides an $L^\infty$ (or pointwise) error estimate which is much stronger than the $L^2$ error estimate in $(\ref{est.ue-u0})$. This $L^\infty$ estimate will play an essential role in the estimation of nodal sets.
\end{remark}

Let $B$ be a ball and $u_0$ be a $C^1$ function in $2B$. In order to show some quantitative stratification results for $u_0$ and $\nabla u_0$, we introduce the doubling index:
\begin{align}\label{def.NuQ}
N(u_0, B)=\log_2 \frac{\sup_{2B}|u_0|}{\sup_{B}|u_0|}
\end{align}
and
\begin{align}
N(\nabla u_0, B)=\log_2 \frac{\sup_{2B}|\nabla u_0|}{\sup_{B}|\nabla u_0|}.
\end{align}
If $u_0$ is a weak solution of the equation $\mathcal{L}_0(u_0)=0$, the doubling index for $|u_0|$ and $|\nabla u_0|$ are monotonic in the sense that
\begin{align}
N(u_0, tB)\leq C N(u_0, B)
\end{align}
and
\begin{align}
N(\nabla u_0, tB)\leq C N(\nabla u_0, B),
\end{align}
for $t\leq \frac{1}{2}$ and $C$ depending only on $d$. This follows from (\ref{est.phi.u0}) and the line after it.

We also define a variant of the above doubling index for cubes. For a cube $Q$,
denote by $s(Q)$ the side length of $Q$.  Define the doubling index in the cube $Q$ by %\textcolor{red}{Do we need to change definition? I don't understand Carlos' thought. But I believe they are comparable}
 %\textcolor{blue}{(For Carlos, I think the definition of doubling index of cubes (\ref{indexvv}), (\ref{indexvv3})  with $B_{2r}$ also works for the later results. It only results in a change of $C$ depending on $d$ in (\ref{monodouble1}), (\ref{monodouble2}). In my opinion,  they use $B_{10dr}$ in \cite{LM18-2} in order to make the comparison of balls and cubes easier, because it happens that $B_r\subset Q_{2\sqrt{d}r}\varsubsetneq B_{2r}$. Let $\sup_{q}|u_0|=|u_0(x_0)|$ and $\sup_{Q}|u_0|=|u_0(x_1)| $. if we choose ball $B_r$ with similar size of $q$, by iterations of $\log \frac{ s(Q)}{s(q)}$ times to get from $x_0$ to $x_1$, we can achieve (\ref{monodouble1}), (\ref{monodouble2}). If we choose $B_{10dr}$, we have to make much changes for later calculations. Especially, these balls will depend on $d$. )}
\begin{align}
N(u_0, Q)=\sup_{x\in Q, r\leq s(Q)}\log_2 \frac{\sup_{ B_{2r}(x)} |u_0|}{\sup_{ B_{r}(x)} |u_0|}
\label{indexvv}
\end{align}
and
\begin{align}
N(\nabla u_0, Q)=\sup_{x\in Q, r\leq s(Q)}\log_2 \frac{\sup_{ B_{2r}(x)} |\nabla u_0|}{\sup_{ B_{r}(x)} |\nabla u_0|}.
\label{indexvv3}
\end{align}
The doubling index defined in cubes is convenient in the sense that if a cube $q$ is a subset of $Q$, then $N(u_0, q)\leq N(u_0, Q)$. Let $q$ be a subcube of $Q$ and $K=\frac{s(Q)}{s(q)}\geq 2$. Then
\begin{align}
\sup_{q}|u_0|\geq K^{-CN(u_0, Q)}\sup_Q|u_0|,
\label{monodouble1}
\end{align}
where $C$ depends only on $d$.
Similarly, it also holds
\begin{align}
\sup_{q}|\nabla u_0|\geq K^{-CN(\nabla u_0, Q)}\sup_Q|\nabla u_0|.
\label{monodouble2}
\end{align}

The following quantitative stratification for $u_0$ and $\nabla u_0$ is the key ingredient of this section. The idea of the proof originates from Lemmas 3.5 and 5.2 in \cite{LM18-2}.

%\textcolor{red}{I reorganized the following statement to make it more readable. But I think we can improve the statement more. In the proof, the balls are equal-sized with finite overlaps. The point is to estimate the number of balls. Maybe we can state it in a more clear way. We do not have to follow Lin\&Shen strictly.	
%A sample: Suppose $N(u_0, Q)\leq {N_0}$. There exists $C^*>0$ such that if $0<\delta < \exp(-C^*N_0)$, there exists a finite sequence of balls $\{ B_{t_i}(x_i) \}_{i=1}^m$
%such that
%\begin{align}
%G_\delta=\Big\{ x\in \frac{1}{2}Q: |u_0(x)|<\delta \sup_Q|u_0(x)| \Big\} \subset \bigcup^m_{i=1} B_{t_i}(x_i),
%\end{align}
%and
%\begin{equation}
%\sum^m_{i=1}t_i^{d-1} \leq C {N_0}^{C},
%\end{equation}
%where $C$ and $C^*$ depend only on $d$.}\textcolor{blue}{I rephrase it}.

\begin{lemma} \label{strati}
	Assume that $u_0$ is harmonic in $5Q$. %\textcolor{red}{(It seems that $3Q$ is enough for the defintion of $N(u_0,Q)$),} \textcolor{blue}{(I think $5Q$ is better, because there are extension of $2Q$ in two sides)}
	\begin{enumerate}
		\item Suppose $N(u_0, Q)\leq {N_0}$. If $0<\delta < \exp(-C^*N_0)$ for some $C^*>0$,  there exists a finite sequence of balls $\{ B_{t_i}(x_i) \}_{i=1}^m$
		such that
		\begin{align}
		G_\delta=\Big\{ x\in \frac{1}{2}Q: |u_0(x)|<\delta \sup_Q|u_0(x)| \Big\} \subset \bigcup^m_{i=1} B_{t_i}(x_i)
		\label{set1}
		\end{align}
		and
		\begin{equation}\label{est.sum.ti}
		\sum^m_{i=1}t_i^{d-1} \leq C {N_0}^{C} s(Q)^{d-1},
		\end{equation}
		where $C$ and $C^*$ depend only on $d$.

		\item Suppose $N(\nabla u_0, Q)\leq {\hat{N}_0}$. If $0<\hat{\delta}\leq  e^{-C \hat{N}_0^3}$ for some $C>0$ depending on $d$,
 there exists a finite sequence of balls $\{ B_{\hat{t}_j}(x_j) \}_{j=1}^{\hat{m}}$ such that
		\begin{align}
		\hat{G}_{\hat{\delta}}=\Big\{ x\in \frac{1}{2}Q: |\nabla u_0(x)|<\hat{\delta} \sup_Q|\nabla u_0(x)|\Big\} \subset \bigcup^{\hat{m}}_{j=1} B_{\hat{t}_j}(x_j)
		\label{set2}
		\end{align}
		and
		\begin{equation}\label{est.sum.hattj}
		\sum^{\hat{m}}_{j=1}\hat{t}_j^{d-1} \leq \frac{1}{4}(\frac{s(Q)}{4})^{d-1}. \quad
		\end{equation}
	\end{enumerate}
\end{lemma}

\begin{proof} In the following proof, all the constants $C, C^*, C_0, C_1,\cdots, C_{14}$ depend only on $d$, and $N_0$, $\hat{N}_0$ are large constants.

	(1) Let $K=\delta^{-\frac{\tau}{{N_0}}}$ and $\delta\leq e^{-\frac{C_0 N_0}{\tau}}$, where $\tau$ is small  to be specified later. We can assume that $K$ is an integer and $K\geq 8$. We divide the cube $\frac{1}{2}Q$ into $K^d$ equal subcubes $q_i$.
	Then $4q_i\subset Q$. We would like to estimate the number of cubes $q_i$ that intersect $G_\delta$.
	
	Let $q_i$ be a cube with
	$q_i\cap G_\delta\not =\emptyset$. Thus, we have $\inf_{q_i}|u_0|<\delta \sup_Q|u_0|$. We claim that if $\delta< e^{-C^* N_0}$ for some large $C^*>1$, then $u_0$ changes sign in $2q_i$. Assume that $u_0$ does not change sign in $2q_i$. By the Harnack inequality,
	\begin{align}
	\sup_{q_i}|u_0|\leq C_1 \inf_{q_i}|u_0|\leq C_1\delta \sup_Q|u_0|.
	\end{align}
	On the other hand, by the monotonicity of the doubling index in cubes (\ref{monodouble1}),
	\begin{align}
	\sup_{q_i}|u_0|\geq C_3 K^{-C_2{N_0}}\sup_{Q}|u_0|= C_3 \delta^{C_2\tau}\sup_{Q}|u_0|.
	\end{align}
	Choosing $\tau=\frac{1}{2C_2}$, we reach a contradiction if \begin{align*}\delta^{\frac{1}{2}}\leq \min \bigg\{ \frac{C_3}{C_1}, \ e^{-{C_0 C_2N_0}} \bigg\}. \end{align*}
	Since $N_0\ge 1$, the last inequality holds if we choose $\delta < e^{-{C^\ast N_0}}$. This proves the claim.
	%Thus, if $q_i$ intersects $G_\delta$, then $u_0$ changes signs in $2q_i$.
	Hence, there are zeros in each $2q_i$ and $q_i\cap G_\delta\not =\emptyset$ with
	\begin{align}
	G_\delta \subset \bigcup^m_{i=1} q_i.
	\end{align}
	This implies (\ref{set1}) as we may replace $q_i$ by $B_{t_i}(x_i)$ with the same center and $t_i=\frac{ s(q_i)\sqrt{d}}{2}$.
	
%	\textcolor{red}{If we are working on harmonic functions (analytic coefficients), do we have to use Logunove's results in \cite{L18-1,L18-2}? Can we just use earlier classical results, if there is any?\textcolor[rgb]{0.00,0.50,1.00}{ {Donnelly-Fefferman's results are on eigenfunctions. We had better use Logunov' results}} }

	Next, to show the first part of (\ref{est.sum.ti}), we need to estimate the number $m$ of the cubes $q_i$. Recall that $4q_i\subset Q$ and each point in $Q$ may be covered by at most a finite number of $4q_i$. By the lower bound estimate of nodal sets in \cite{L18-2}, we have
	\begin{align}\label{est.lower.InQ}
	H^{d-1}(\{u_0=0\}\cap Q)\geq C\sum_{i=1}^m H^{d-1}(\{u_0=0\}\cap 4q_i)\geq C_4 m \Big(\frac{s(Q)}{K} \Big)^{d-1}.
	\end{align}
	On the other hand, by the upper bound estimate of nodal sets in \cite{L18-1}, it holds that
	\begin{align}\label{est.upper.InQ}
	H^{d-1}(\{u_0=0\}\cap Q)\leq C_5 N_0^{C_6} s(Q)^{d-1},
	\end{align}
where $C_6>\frac{1}{2}$.
	Combining (\ref{est.lower.InQ}) and (\ref{est.upper.InQ}), we arrive at
	\begin{align}
	C_4 m \Big(\frac{s(Q)}{K}\Big)^{d-1}\leq C_5{N}_0^{C_6} s(Q)^{d-1},
	\end{align}
	which yields
	\begin{align}
	m\leq \frac{C_5}{C_4} {N}_0^{C_6} K^{d-1}.
	\end{align}
	Thus,
	\begin{align}
	\sum_{i=1}^m t_i^{d-1} = C_d m\cdot s(q_i)^{d-1} \leq C{N}_0^{C_6} K^{d-1} \Big(\frac{s(Q)}{K}\Big)^{d-1}\leq C {N}_0^{C_6} s(Q)^{d-1}.
	\end{align}
	This proves (1).
	%Since the cubes and the balls are comparable, we can replace $q_i$ by the ball $B_{t_i}$ at the same center with $t_i=\frac{ s(q_i)\sqrt{d}}{2}$.
	
	(2) Next, we establish the estimates (\ref{set2}) and (\ref{est.sum.hattj}).
	We divide the cube $\frac{1}{2}Q$ into $K_1^d$ subcubes with side length $\frac{s(Q)}{2K_1}$. The size of $K_1$, depending on $\hat{\delta}$, will be chosen later. The cube $q_j$ is called bad if
	\begin{align}\label{est.bad}
	\inf_{q_j}|\nabla u_0|\leq c\sup_{2q_j}|\nabla u_0|
	\end{align}
	for some small $c$ depending only on $d$. We claim that the number of bad cubes $q_j$ is not greater than $e^{C_d \hat{N}_0^2} K_1^{d-2}$, where $C_d$ depends on $d$.
	
	To show the above claim, we need to use \cite[Theorem 1.1]{NV17}. Recall the \textit{effective} critical set is defined as
	\begin{equation*}
	\mathcal{C}_r(u_0) = \bigg\{ x\in Q: \inf_{B_r(x)} r^2|\nabla u_0|^2 \le \frac{d}{16} \fint_{\partial B_{2r}(x)} (u - u(x))^2  \bigg\}.
	\end{equation*}
	Let $B_r(\mathcal{C}_r(u_0))$ be the $r$-neighborhood of $\mathcal{C}_r(u_0)$, namely, $B_r(\mathcal{C}_r(u_0)) = \{ x\in Q: \txt{dist}(x, \mathcal{C}_r(u_0)) < r \}$. Then \cite[Theorem 1.1]{NV17} implies
	\begin{equation}\label{est.Vol.BCr}
	|B_r(\mathcal{C}_r(u_0)) \cap B_s | \le C^{\big( \widetilde{N}(u_0,B_{2s} ) \big)^2} \Big( \frac{r}{s} \Big)^2 |B_s|,
	\end{equation}
	where $B_s, B_{2s}$ are concentric balls such that $B_{4s} \subset Q$ and $\widetilde{N}$ is the modified frequency function defined by
	\begin{equation*}
	\widetilde{N}(u_0,B_{2s}) := \frac{2s \int_{B_{ 2s}} |\nabla u_0|^2}{ \int_{ \partial B_{2s}} ( u_0 - { u_0(z)} )^2 },
	\end{equation*}
	where $z$ is the center of $B_s$.
	By \cite[Corollary 2.2.6]{HL13} and the mean value property of harmonic functions, we have
	\begin{equation*}
	\widetilde{N}(u_0,B_{2s}) \le C \log_2 \frac{\int_{B_{4s}} (u_0 - {u_0(z)})^2 }{\int_{B_{2s}} (u_0 -  u_0(z))^2  } \le C\log_2 \frac{ \sup_{B_{4s}}|\nabla u_0| }{\sup_{B_{s}}|\nabla u_0|} \le CN(\nabla u_0,Q) \le C\hat{N}_0,
	\end{equation*}
	where we have also used a gradient estimate for harmonic functions in the second inequality.
	Hence, (\ref{est.Vol.BCr}) implies
	\begin{equation}\label{est.BrCr.hatN}
		|B_r(\mathcal{C}_r(u_0)) \cap B_s | \le C^{\hat{N}_0^2 } \Big( \frac{r}{s} \Big)^2 |B_s|.
	\end{equation}
	Next, we show that if $q_j$ is a bad cube with sufficiently small $c$, then $q_j \cap \mathcal{C}_r (u_0) \neq \emptyset.$ Actually if $q_j$ is bad and $x_j$ is the point in $q_j$ so that $|\nabla u_0(x_j)| = \inf_{{q_j}} |\nabla u_0|$, then
	\begin{equation*}
	\inf_{B_r(x_j)} |\nabla u_0| \le |\nabla u_0(x_j)| = \inf_{q_j} |\nabla u_0| \le c\sup_{2q_j}|\nabla u_0|,
	\end{equation*}
	where we used the condition (\ref{est.bad}) in the last inequality. Fix $r =2\sqrt{d} s(q_j)$. Then $2q_j \subset B_r(x_j)$. It follows from the gradient estimate and the Caccioppoli inequality that
	\begin{equation*}
	\begin{aligned}
	\inf_{B_r(x_j)} r |\nabla u_0| \le c r \sup_{B_r(x_j)}|\nabla u_0|& \le c C_d r \bigg( \fint_{B_{\frac32 r}(x_j)} |\nabla u_0|^2 \bigg)^{1/2} \\
	 & \le c C_d^2  \bigg( \fint_{B_{2r}(x_j)} | u_0 - u_0(x_j) |^2 \bigg)^{1/2}.
	\end{aligned}
	\end{equation*}
	In view of \cite[Corollary 2.2.7]{HL13}, we have
	\begin{equation*}
	\begin{aligned}
	\inf_{B_r(x_j)} r |\nabla u_0| & \le \frac{c C_d^2 }{d} \bigg( \fint_{\partial B_{2r}(x_j)} | u_0 - u_0(x_j) |^2 \bigg)^{1/2} \\
	& \le \sqrt{\frac{d}{16}}\bigg( \fint_{\partial B_{2r}(x_j)} | u_0 - u_0(x_j) |^2 \bigg)^{1/2},
	\end{aligned}
	\end{equation*}
	where in the last inequality, we choose $c$ small so that $c C_d^2/d < \sqrt{d/16}$. This implies that $x_j \in \mathcal{C}_r(u_0)$ and $q_j\cap \mathcal{C}_r(u_0) \neq \emptyset$. Because $r = 2\sqrt{d} s(q_j)$, we have $q_j \subset B_r(\mathcal{C}_r(u_0))$. This means that all the bad cubes $q_j$ are contained in $B_r(\mathcal{C}_r(u_0))$. Finally, let $s$ be comparable to $s(Q)$ and note that $\frac12 Q$ can be covered by finitely many, depending only on $d$, $B_s$ with $B_{4s} \subset Q$. Then, by (\ref{est.BrCr.hatN}), the total volume of bad cubes in $\frac12 Q$ is bounded by $C^{\hat{N}_0^2} \big( s(q_j)/s(Q) \big)^2 |Q| \le C^{\hat{N}_0^2} K_1^{-2} |Q|$. Hence, the number of bad cubes is not greater than $C^{\hat{N}_0^2} K_1^{d-2}$. The claim has been proved.

	Now, for any $q_j$, the monotonicity of the doubling index of $\nabla u_0$ in cubes in (\ref{monodouble2}) shows that
	\begin{align}
	\sup_{q_j}|\nabla u_0| \geq C_8 K_1^{-C_7\hat{N}_0} \sup_{Q}| \nabla u_0|.
	\end{align}
	If $q_j$ is not bad, the reverse inequality of (\ref{est.bad}) yields
	\begin{align}\label{est.notbad}
	\inf_{q_j}|\nabla u_0| > C_9K_1^{-C_7\hat{N}_0}\sup_{Q}| \nabla u_0|.
	\end{align}
	
	Given $\hat{\delta}$, small enough (to be quantified later), we want to estimate the set $\hat{G}_{\hat{\delta}}$ defined in (\ref{set2}).
	If $q_j$ is not bad and we choose  $K_1$ to be the smallest integer such that
	\begin{align}
	C_9(K_1+1)^{-C_7\hat{N}_0}<  \hat{\delta},
	\end{align} then (\ref{est.notbad}) gives
	\begin{equation*}
	\inf_{q_j}|\nabla u_0|\geq C_9K_1^{-C_7\hat{N}_0}\sup_{Q}| \nabla u_0|>\hat{\delta}\sup_{Q}| \nabla u_0|.
	\end{equation*}
	This implies that $q_j$ does not intersect $\hat{G}_\delta$. It  also shows that $K_1\approx \hat{\delta}^{\frac{-1}{C_{10}\hat{N}_0}}$.
	Thus, the set $\hat{G}_{\hat{\delta}}$ is covered by the union of bad cubes of size $\frac{s(Q)}{2K_1}$. Again, we may now replace bad $q_j$ by $B_{\hat{t}_j}(x_j)$ with the same center and $\hat{t}_j = \frac{ s(q_j)\sqrt{d}}{2}$. Let $\hat{m}$ be the number of bad cubes and recall that $e^{C_d \hat{N}_0^2} K_1^{d-2} \ge \hat{m}$.
	It follows that
	\begin{align}
	\sum^{\hat{m}}_{j=1} \hat{t}_j^{d-1} = C \hat{m}\cdot s(q_j)^{d-1} &\leq  C e^{C_d \hat{N}_0^2} K_1^{d-2} \bigg(\frac{s(Q)}{2K_1}\bigg)^{d-1}\nonumber \\
	&\leq \Big(\frac{1}{2}\Big)^{d-1}s^{d-1}(Q) C e^{C_d \hat{N}_0^2} K_1^{-1}\nonumber \\
	&\leq \Big(\frac{1}{2}\Big)^{d-1}s^{d-1}(Q) C e^{C_d \hat{N}_0^2} {\hat{\delta}}^{\frac{1}{C_{10}\hat{N}_0}} \nonumber \\
	&\leq \Big(\frac{1}{4}\Big)^{d} s^{d-1}(Q)
	\end{align}
	where  we have chosen $\hat{\delta}\leq  e^{-C_{11} \hat{N}_0^3}$ in the last inequality. This completes the proof.
	%Similarly, we can replace $q_j$ by the ball $B_{\hat{t}_j}$ at the same center with $\hat{t}_j=\frac{s(q_j)\sqrt{d} }{2}$.
\end{proof}

Next we estimate the density function $E_\e(y, r)$ of nodal sets, which is the initial step for an iterative argument to obtain Theorem \ref{mainth}. The following lemma is a quantitative version of \cite[Lemma 4.5]{LS19}.
Without loss of generality, we may identify $Q$ in Lemma \ref{strati} by $B_\frac{1}{8}$.

\begin{lemma}
	Let $\beta\in (\frac{3}{4}, 1)$ and (\ref{NewPrecond}) hold. If $\e\leq \exp(-C(\ln N)^3)$, then  there exists a finite sequence of balls $\{ B_{\hat{t}_j}(y_j): j= 1\cdots, \hat{m}\}$ such that $y_j\in B_{\frac{1}{16}}$, $\hat{t}_j \in (0, \frac{1}{128})$ and
	\begin{align}
	E_\e(0, \frac{1}{16})\leq \exp(CN^\frac{2}{\beta-\frac{3}{4}})+\frac{1}{4}\sup_{1\le j \le \hat{m}} E_\e(y_j, \hat{t}_j),
	\label{keyineq}
	\end{align}
	where $c$ and $C$ depend on $d$, $\Lambda$ and $\gamma$.
	\label{keylemma}
\end{lemma}
\begin{proof}
	We assume $E_{\e}(0,\frac{1}{16})>0$. Otherwise, (\ref{keyineq}) is trivial. We will make use of the approximation estimate by a harmonic function in Lemma \ref{lem.1step}. Using (\ref{LSS}) with $r =1$, we have \begin{align}
	\|u_\e-u_0\|_{L^\infty(B_\frac{1}{8})}\leq \hat{C}\e \|u_\e\|_{L^2(B_2)}.
	\label{compare}
	\end{align}
	
	By normalization, we may assume that $\int_{B_{2}} u_\e^2 =1$. We would like to estimate the doubling index for $u_0$ and $\nabla u_0$. From (\ref{noreal}) and convexity of $\varphi$ in (\ref{convex}), we have %\textcolor{red}{(Is the constant $2\sqrt{N}$, instead of $4N$?)}\textcolor{blue}{(you are right. I enlarge it at the early version)}
	\begin{align}
	\|u_0\|_{L^2(B_{\frac{3t}{2}})}
	&\leq 2\sqrt{N} \|u_0\|_{L^2(B_{{t}})}
	\label{haodouble}
	\end{align}
	for all $0<t<1$.
	By elliptic estimates and using (\ref{haodouble}) twice,
	\begin{equation}\label{est.uDu.Linfty}
	\begin{aligned}
	\|u_0\|_{L^\infty(B_1)} + \|\nabla u_0\|_{L^\infty(B_1)} & \leq C \bigg(\int_{B_\frac{9}{8}} u^2_0 \bigg)^\frac{1}{2} \\
	& \leq C\sqrt{N}  \bigg(\int_{B_{\frac{3}{4}}} u^2_0 \bigg)^\frac{1}{2} \leq C{N}  \bigg(\int_{B_\frac{1}{2} } u^2_0 \bigg)^\frac{1}{2} \le CN\| u_0 \|_{L^\infty(B_{\frac12})}.
	\end{aligned}
	\end{equation}
	The above estimate includes
	\begin{equation}
	\| u_0\|_{L^\infty({B_{1}})} \leq CN \| u_0\|_{L^\infty({B_{\frac{1}{2}}})} .
	\label{inftymono}
	\end{equation}
Thus, $N(u_0, B_{\frac{1}{2}})\leq \log_2 {(CN)}$.
Similarly to the frequency function in Section 3, one can define the frequency function introduced for the harmonic function $u_0$ as
\begin{align*}
\mathcal{N}(x, r)=\frac{r \int_{ B_r(x)}|\nabla u_0|^2}{\int_{ \partial B_r(x)} u_0^2  }.
\end{align*}
By Theorem 2.2.8 in \cite{HL13}, it holds that
\begin{align}\mathcal{N}(x, \ \frac{1}{2}(\frac{1}{2}-R))\leq C\mathcal{N}(0, \frac{1}{2})
\label{HLmono}
\end{align}
for any $x\in B_R$ and $ 0<R<\frac{1}{2}$. It is known that the
doubling index and the frequency function are comparable.
It follows from Lemma 7.1 in \cite{L18-1} that
\begin{align}
\mathcal{N}(0, \frac{1}{2})\leq C N(u_0, B_{\frac{1}{2}})\leq C\log_2 {(CN)},
\end{align}
where $C$ depends on $d$. From (\ref{HLmono}), let $R=\frac{1}{8}$, we obtain that
\begin{align*}
\mathcal{N}(x, \frac{3}{16})\leq C\log_2 {(CN)}.
\end{align*}
Using Lemma 7.1 in \cite{L18-1} again, we see that
\begin{align}
 N(u_0, B_{r}(x))\leq C\mathcal{N}(x, \frac{3}{16})\leq C\log_2 {(CN)}
\end{align}
for any $0<r\leq \frac{1}{8}.$
Thus, we obtain that
	\begin{align}
	\| u_0\|_{L^\infty({B_{2r}(x)})}  \leq CN^{C_0} 	\| u_0\|_{L^\infty({B_{r}(x)})}
	\label{detailmono1}
	\end{align}
	for any $0<r\leq \frac{1}{8}$ and any $x\in B_{\frac{1}{8}}$, where $C_0$ depends on $d$.

%\textcolor{red}{(When you shift the center, do you need a power of $N$? Also, does the monotonicity holds for $L^\infty$ norm, instead of $L^2$ norm?)}. \textcolor{blue}{ (I have showed the arguments for the calculations. There is indeed a power of $N$.  I thought $N$ was exponential type function. Thanks for pointing it out. It will not affect the later results)}.

	Next, we estimate the doubling index of $\nabla u_0$. We first claim that there are zeros for $u_0$ in $B_\frac{1}{8}$. In fact, from (\ref{compare}) and Theorem \ref{thm.large.r} with $\theta = 1/2$, we have %\textcolor{red}{I change $N$ to $N^2$ because we use the doubling inequality twice}
	\begin{align}
	\|u_\e-u_0\|_{L^2(B_\frac{1}{8})}\leq {C}{N^2}\e \|u_\e\|_{L^2(B_\frac{1}{8})}.
	\end{align}
	Then
	\begin{align}
	\|u_0\|_{L^2(B_\frac{1}{8})}&\geq (1-{C}{N^2}\e) \|u_\e\|_{L^2(B_\frac{1}{8})} \nonumber \\
	&\geq \frac{(1-{C}{N^2}\e)}{C{N^2}} \|u_\e\|_{L^2(B_2)}.
	\end{align}
	It follows that
	$\|u_0\|_{L^2(B_\frac{1}{8})}\geq \frac{1}{CN^2}$ if $\e< cN^{-{2}}$. Hence, (\ref{unbelieve}) implies %\textcolor{red}{(How did you get $L^\infty$ estimate from $L^2$ without any loss on the constant?)}\textcolor{blue}{(It is true using the arguments as the estimating the gradients below. In order not to have much changes, we can use $L^2$ norm. )  }
	\begin{align}\label{est.Linfty.14}
	\frac{1}{CN^2}\leq \|u_0\|_{L^2(B_\frac{1}{8})}\leq CN\|u_0\|_{L^2(B_\frac{1}{16})}.
	\end{align}
	Now, let us assume that $u_0$ has no zeros in $B_{\frac18}$ and therefore does not change signs in $B_{\frac18}$. Without loss of generality, we may assume that $u_0$ is positive. By the Harnack inequality and (\ref{est.Linfty.14}),
	\begin{align}
	\inf_{B_\frac{1}{16}}|u_0|\geq C \sup_{B_\frac{1}{16}}|u_0|\geq \frac{1}{CN^3}.
	\end{align} From (\ref{compare}), for $x\in B_\frac{1}{16}$,  we get
	\begin{align}
	\frac{1}{CN^3}-\hat{C}\e \leq u_\e(x).
	\end{align}
	Since $\e<CN^{-3}$, then $u_\e(x)>0$ for $x\in B_\frac{1}{16}$. This contradicts our assumption that $E_\e(0,\frac{1}{16}) > 0$. Thus, the claim has been shown.
	
%	Now, since there are zeros for $u_0$ in $B_\frac{1}{8}$,
%	by the mean-value theorem, we have
%	\begin{align} \sup_{B_\frac{1}{8}}|\nabla u_0|>C\|u_0\|_{L^\infty(B_\frac{1}{8})} >\frac{\tilde{C}}{N^2},
%	\label{magra}.
%	\end{align}

	Now, since $u_0$ has zeros in $B_{\frac18}$ (and hence in $B_{\frac12}$), we obtain from (\ref{est.uDu.Linfty}) and the mean value theorem that
	\begin{align}
	\sup_{B_{1}}|\nabla u_0| \leq CN \sup_{B_\frac{1}{2}}|\nabla u_0|.
	\label{gradmono}
	\end{align}
	Again, by the relation of the frequency function for $|\nabla u_0|$ and the doubling index, we can argue as the derivation of (\ref{detailmono1}) that (\ref{gradmono})
implies that
	\begin{align}
	\sup_{B_{2r}(x)}|\nabla u_0| \leq CN^{C_0} \sup_{B_r(x)}|\nabla u_0|
\label{detailmono}
	\end{align}
	for any $0<r\leq \frac{1}{8}$ and any $x\in B_{\frac{1}{8}}$. %\textcolor{red}{(When you shift the center, do you need a power of $N$?)}.\textcolor{blue}{(It's argument is similar to the case for $u_0$)}
 %\textcolor{red}{(Do (\ref{est.Linfty.14}) and the above estimate imply the estimate of $N(u_0,Q)$ and $N(\nabla u_0, Q)$?)} \textcolor{blue}{ (I provide more details, the (\ref{detailmono}) and (\ref{detailmono1}) give the estimate of $N(u_0,Q)$ and $N(\nabla u_0, Q)$. shrink the $Q$ if necessary.})
% Using (\ref{haodouble}) and the arguments to derive (\ref{gradmono}), we can show that
% \begin{align}
%	\| u_0\|_{L^\infty({B_{1}})} \leq CN \| u_0\|_{L^\infty({B_{\frac{1}{2}}})} .
%	\label{inftymono}
%	\end{align}
%From the monotonicity property of the double index of $u_0$, we obtain that

	Thanks to the monotonicity of the doubling index for $u_0$ and $\nabla u_0$,
	from the definition of $N(u_0, Q)$ and $N(\nabla u_0, Q)$, we know that $N(u_0, Q) \le {N}_0:= C \log N$ and $N(\nabla u_0, Q) \le \hat{N}_0:= C \log N$.
	In order to apply Lemma \ref{strati}, we assume that $\hat{C}\e\leq \delta\leq e^{-C^\ast C\log N}$ and
	$\frac{\hat{\delta}}{2}\approx  \frac{1}{2} e^{-C(\log N)^3}$. Thus, we require $\e\leq CN^{-\alpha}$ for some $\alpha$ depending on $d$, which is satisfied by the assumption of $\e$ in the lemma.
	%We choose $\alpha=\frac{1}{\beta-\frac{3}{4}}$ for some $ \frac{3}{4}<\beta<1$. Note that $\alpha>4$.
	With the aid of (\ref{compare}), we have
	\begin{align}
	Z(u_\e)\cap B_{\frac{1}{16}} &\subset Z(u_\e)\cap \{x\in B_{\frac{1}{16}}: |u_0(x)|\leq \hat{C}\e\} \nonumber \\
	&\subset Z(u_\e)\cap \{x\in {B_{\frac{1}{16}}}: |u_0(x)|\leq \hat{C}\e \ \mbox{and} \ |\nabla u_0(x)|\geq \frac{\hat{\delta}}{2}\sup_{{B_{\frac{1}{8}}}}|\nabla u_0|\} \nonumber \\ & \bigcup
	Z(u_\e)\cap \{x\in {B_{\frac{1}{16}}}: |u_0(x)|\leq \hat{C}\e \ \mbox{and} \ |\nabla u_0(x)|\leq \frac{\hat{\delta}}{2}\sup_{B_{\frac{1}{8}}}|\nabla u_0|\} \nonumber \\
	&\subset \big( \cup^m_{i=1}Z(u_\e)\cap G_i\big) \bigcup \big( \cup^{\hat{m}}_{j=1}Z(u_\e)\cap B_{\hat{t}_j}(y_j) \big),
	\end{align}
	where
	\begin{align}
	G_i=\bigg\{ x\in B_{{t}_i}(x_i)| \ |u_0(x)|\leq \hat{C}\e \ \mbox{ and } \ |\nabla u_0(x)|\geq \frac{\hat{\delta}}{2}\sup_{B_{\frac{1}{8}}}|\nabla u_0| \bigg\},
	\end{align}
	and $B_{\hat{t}_j}(y_j)$ and $B_{t_i}(x_i)$ are given by Lemma \ref{strati}.
	Thus, it follows from Lemma \ref{strati}  that
	\begin{align}
	H^{d-1}(Z(u_\e)\cap {B_{\frac{1}{16}}})&\leq \sum^m_{i=1} H^{d-1}(Z(u_\e)\cap G_i  )+\sum^{\hat{m}}_{j=1} H^{d-1}(Z(u_\e)\cap B_{\hat{t}_j}(y_j)  )\nonumber \\
	&\leq \bigg( \sup_{i} \frac{H^{d-1}(Z(u_\e)\cap G_i  )}{t_i^{d-1}} \bigg) \sum^m_{i=1} t_i^{d-1}+\sup_{j} E_\e(y_j, \hat{t}_j) \sum^{\hat{m}}_{j=1} \hat{t}_j^{d-1} \nonumber \\
	&\leq C(\frac{1}{16})^{d-1}(\log N)^{\hat{\alpha}} \sup_{i} \frac{H^{d-1}(Z(u_\e)\cap G_i  )}{t_i^{d-1}}+\frac{1}{4} (\frac{1}{16})^{d-1} \sup_{j} E_\e(y_j, \hat{t}_j),
	\label{import}
	\end{align}
where $\hat{\alpha}$ depends on $d$.
	Since $N$ is large, by the decomposition in Lemma \ref{strati}, we may assume $0< \hat{t}_j<\frac{1}{128}$. %\textcolor{red}{(Why $1/128$?)}\textcolor{blue}{( I want $1/128=\frac{1}{16}\frac{1}{8}$, in the latter lemma, I can do the iteration to get small radius)}.
	
	Next we estimate the upper bound for $H^{d-1}(Z(u_\e)\cap G_i)$ for each $i$. We will discuss two cases $t_i\leq \hat{C}\e$ and
	$t_i\geq \hat{C}\e$. If $t_i\leq \hat{C}\e$, by  Lemma \ref{smallscale},
	\begin{align}
	H^{d-1}(Z(u_\e)\cap G_i)&\leq H^{d-1}(Z(u_\e)\cap B_{t_i}(x_i)) \nonumber \\
	&\leq C \exp(CN^\frac{2}{\beta-\frac{3}{4}}) t_i^{d-1}.
	\label{smallt}
	\end{align}
	
	Now, we consider the case $t_i\geq \hat{C}\e$. Note that (\ref{est.Linfty.14}), the fact that $u_0$ has zeros in $B_{\frac18}$ and the definition of $G_i$ imply
	\begin{equation}\label{est.Gi.Du0}
	|\nabla u_0(x)|\geq \frac{\hat{\delta}\tilde{C}}{2N^2}, \quad \text{and} \quad |u_0(x)| \le \hat{C}\e
	\end{equation}
	for any point $x\in G_i$. 	
	%Thus, by the $C^2$ regularity of $u_0$, we know that there is $c>0$ such that $|\nabla u_0(x)| \ge\frac{\hat{\delta}c}{2N^2} $ if the distance from $x$ to the set $G_i$ is less than $\frac{\hat{\delta}c}{2N^2}$. Assume $\e < \frac{\hat{\delta}^2 c^2}{N^4}$ (This requirement leads to the assumption $\e < \exp(-C(\ln N)^3)$). Now, the implicit function theorem implies that only in a relatively small portion of $B_{t_i}(x_i)$, the function $u_0$ can be less than $\hat{C}\e$. Precisely, there are finite balls $\{ B_{C\e}(z_l): l = 1,2,\cdots, m_2\}$ covering $G_i$ and $\textcolor[rgb]{1.00,0.50,0.50}{m_2 \e^d \le C \e N^4 \hat{\delta}^{-2} t_i^{d}}$. (We may add more details if necessary.)
	Fix $i$. For $k = 1,2,\cdots, d$, define
	\begin{equation*}
	F_k^{\pm} = \Big\{ x\in B_{t_i}(x_i)  \Big|\ |u_0(x)| \le \hat{C} \e, \ \pm \frac{\partial u_0}{ \partial x_k}(x) \ge \frac{\hat{\delta} c }{2dN^2} \Big\}.
	\end{equation*}
	Then (\ref{est.Gi.Du0}) implies that $G_i$ is contained in $\cup_{k =1}^d (F_k^+ \cup F_k^{-})$.
	%{\color{blue}(For Carlos: we deleted ``if $t_i\geq \hat{C}\e$'', since it only indicates for the case we are considering. Thank you for careful reading and making it clear.)}
	Without loss of generality, it suffices to estimate $F_k^+$. By the $C^2$ regularity of $u_0$, for any $x_0\in F_k^+$, there exists a cylinder $\mathcal{C}(x_0)$ centered at $x_0$, whose base is a square perpendicular to $e_k$ with side length $C\e$, such that the height of $\mathcal{C}(x_0)$ is $\frac{\hat{\delta} c_1 }{2dN^2}$ and
	\begin{equation}\label{est.Dku0+}
	\frac{\partial u_0}{ \partial x_k}(x) \ge \frac{\hat{\delta} c_1 }{2dN^2}, \quad \txt{for any } x\in \mathcal{C}(x_0),
	\end{equation}
	where $c_1>0$ is a constant smaller than $c$.

	\begin{figure}[h]
		\begin{center}
			\includegraphics[scale =0.4]{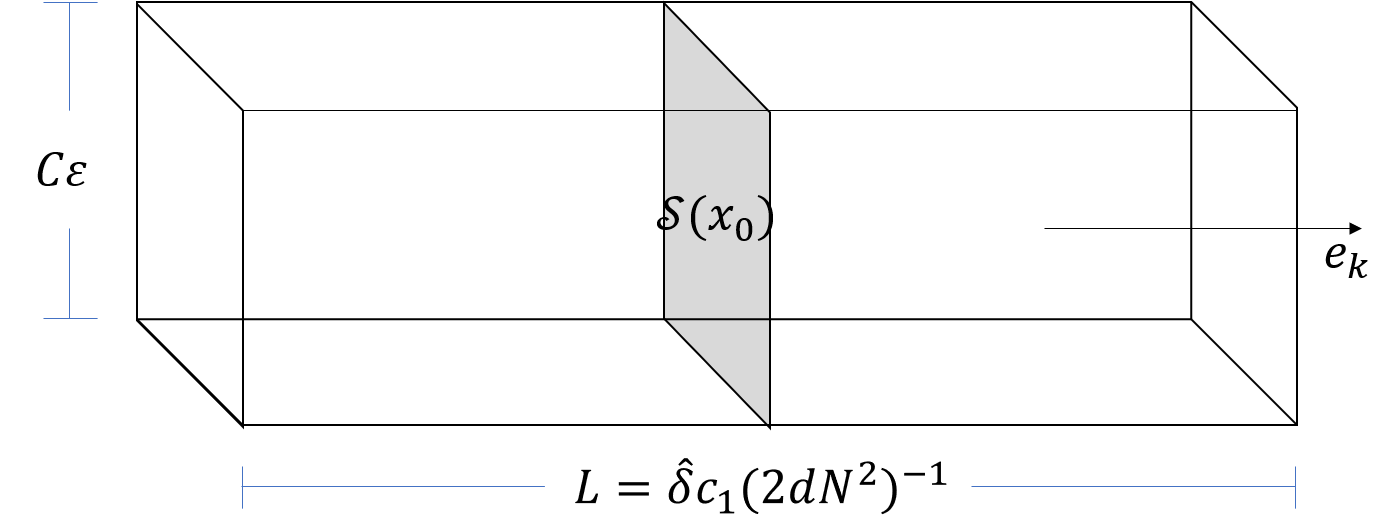}
		\end{center}
		\vspace{-1 em}
		\caption{The cylinder $ \mathcal{C}(x_0)$}\label{fig_1}
	\end{figure}

	We would like to show that $\mathcal{C}(x_0) \cap F_k^+$ can be covered by $m_1$ balls with radius $C\e$, where $m_1 \le \frac{C_2 2dN^2 }{\hat{\delta}}$. Let $\mathcal{S}(x_0)$ be the cross section containing $x_0$ of the cylinder $\mathcal{C}(x_0)$ which is perpendicular to $e_k$. Since $|\nabla u_0| \le C$ and $|u_0(x_0)| \le \hat{C}\e$, we see that $|u_0(y)| \le C_1\e$ for any $y\in \mathcal{S}(x_0)$. Next, because of (\ref{est.Dku0+}), for any $y \in \mathcal{S}(x_0)$ and $t>0$,
	\begin{equation*}
	u_0(y+te_k) \ge t \frac{\hat{\delta} c_1 }{2dN^2} - C_1 \e.
	\end{equation*}
	This implies that $y+ t e_k \notin F_k^+$ if $t > \frac{(C_1+\hat{C}) \e 2dN^2}{\hat{\delta} c_1}$. Similarly, $y+ t e_k \notin F_k^+$ if $t < - \frac{(C_1+\hat{C}) \e 2dN^2}{\hat{\delta} c_1}$. This implies that
	\begin{equation*}
	F_k^+ \cap \mathcal{C}(x_0) \subset \Big\{ y+te_k | \ y\in \mathcal{S}(x_0), \ |t| \le \frac{(C_1+\hat{C}) \e 2dN^2}{\hat{\delta} c_1} \Big\}.
	\end{equation*}
	Consequently, $F_k^+ \cap \mathcal{C}(x_0)$ can be covered by $m_1$ balls with radius $C\e$ and $m_1 \le \frac{C_2 2dN^2 }{\hat{\delta}}$.
	\\
	Now, because $F_k^+ \subset B_{t_i}(x_i)$ can be covered by $m_2$ cylinders, with
	\begin{equation*}
		m_2 \le
		\left\{
		\begin{aligned}
		 	&\frac{C|B_{t_i}|}{|\mathcal{C}(x_0)|} = \frac{Cdt_i^d N^2}{\hat{\delta} \e^{d-1}}, \qquad \text{if } t_i \ge L = \frac{\hat{\delta} c_1}{2dN^2}, \\
			&\frac{Ct_i^{d-1}}{(C\e)^{d-1}} = \frac{Ct_i^{d-1} }{\e^{d-1}}, \qquad \text{if } \hat{C}\e \le t_i \le L = \frac{\hat{\delta} c_1}{2dN^2},
		\end{aligned}
		\right.
	\end{equation*}
	such that the $\mathcal{C}(x_0)$'s have finite overlaps, then $F_k^+$ can be covered by $m$ balls with radius $C\e$ (denoted by $\{ B_{C\e}(z_{k,\ell}^+): \ell = 1,2,\cdots, m \}$) , where
	\begin{equation*}
	m = m_1 m_2 \le \frac{C t_i^{d-1} d^2 N^4}{\e^{d-1} \hat{\delta}^2 }.
	\end{equation*}
	Note that the same estimate holds for $F_k^{-}$ as well for each $k = 1,2,\cdots, d$.
	 Hence, by Lemma \ref{smallscale}, we derive that
	\begin{align}
	H^{d-1}(Z(u_\e)\cap\{x\in B_{t_i}(x_i):|\nabla u_0|\geq \frac{\tilde{C}\hat{\delta}}{2N}\}  )
	&\leq \sum_{k = 1}^d \sum^{m}_{l=1} H^{d-1}(Z(u_\e)\cap B_{C\e}(z_{k,l}^{\pm})) \nonumber \\
	&\leq C \exp(CN^\frac{2}{\beta-\frac{3}{4}}) m \e^{d-1} \nonumber \\
	&\leq C N^4 \hat{\delta}^{-2}  \exp(CN^\frac{2}{\beta-\frac{3}{4}}) t_i^{d-1}\nonumber \\
	&\leq \exp(CN^\frac{2}{\beta-\frac{3}{4}}) t_i^{d-1},
	\end{align}
	where in the last inequality, we have used the fact $\hat{\delta}^{-1} \approx \exp(C(\ln N)^3)$ and enlarged the constant $C$. Note that here $\beta\in (\frac{3}{4},1)$ can be arbitrary.
	%Similarly we can carry the arguments to other regions.
	%Note that $\alpha=\frac{1}{\beta-\frac{3}{4}}$ for some $\frac{3}{4}<\beta<1$.
	Thus, together with (\ref{smallt}), we obtain that
	\begin{align}
	H^{d-1}(Z(u_\e)\cap G_i)\leq \exp(CN^\frac{2}{\beta-\frac{3}{4}}) t_i^{d-1}.
	\end{align}
	Taking (\ref{import}) into account, we arrive at the conclusion (\ref{keyineq}).
\end{proof}

\subsection{Proof of Theorem \ref{mainth}}
Thanks to Lemma \ref{keylemma}, we are able to show the upper bound of the nodal sets of $u_\e$ in the interior domain.
\begin{proof}[Proof of Theorem \ref{mainth}  ($d\ge 3$)]
	We first consider the case $\e\leq \exp(-C(\ln N)^3)$.  Recall from (\ref{comparedouble}) that \begin{equation}
	\int_{B_{5/3}(x)} u_\e^2  \le 2N^2 \int_{B_{5/6}(x)} u_\e^2
	\end{equation}
for any $x\in B_{1/3}$.
By Theorem \ref{thm.large.r}, it follows that
\begin{equation}
	\int_{B_{2r}(x)} u_\e^2  \le 4N^2 \int_{B_{r}(x)} u_\e^2
\label{moutain}
	\end{equation}
for $ CN^{\frac{1}{\beta-\frac{3}{4}}} \e <r<\frac{5}{6}$ and any $x\in B_{1/3}$.
By (\ref{moutain}) and Lemma \ref{lem.1step}, we derive that
\begin{align}
\int_{B_r(x)} u_0^2 \leq CN^4 \int_{B_\frac{r}{2}(x)} u_0^2\ ,
\label{double-mono}
\end{align}
as in (\ref{unbelieve}) for $x\in B_{\frac{1}{3}}$ and $CN^{\frac{1}{\beta-\frac{3}{4}}} \e <r< \frac{5}{6}$.
By examining the proof of Lemma \ref{keylemma}, the estimates (\ref{moutain}) and (\ref{double-mono}) guarantee that the arguments in the Lemma \ref{keylemma} hold for $E(x_0, s)$ for $x_0\in B_{\frac{1}{8}}$ and $C\e \exp(C(\ln N)^3) <  s\leq {\frac{r}{16}}$.

Let $v(x)=u_\e(x_0+tx)$ for any $t$  satisfying $CN^{\frac{1}{\beta-\frac{3}{4}}} \e \le C\e \exp(C(\ln N)^3) <t<\frac{5}{6}$ and $x_0\in B_{1/8}$.
Then $v(x)$ satisfies
\begin{align}
\nabla\cdot ( {A}^{\e, t}_{x_0}(x)\nabla v(x))=0 \quad \mbox{in} \ B_2,
\end{align}
where ${A}^{\e, t}_{x_0}(x)=A(\e^{-1}(x_0+tx))$.
By Lemma \ref{keylemma}, we have
	\begin{align}
	\frac{H^{d-1}(Z(v)\cap B_{\frac{1}{16} }(0))}{(\frac{1}{16})^{d-1}}&\leq \exp(CN^\frac{2}{\beta-\frac{3}{4}}) + \frac{1}{4}\sup_{j}\frac{ H^{d-1}(Z(v)\cap B_{\tilde{s}_j}( y_j))}{( \tilde{s}_j)^{d-1}},
	\label{itera}
	\end{align}
where $\tilde{s}_j\in (0, \frac{1}{16\times 8})$ and $y_j\in  B_{\frac{1}{16} }(0)$. By rescaling, we reduce the estimate to $u_\e$ and obtain that
\begin{align}
	\frac{H^{d-1}(Z(u_\e)\cap B_{\frac{t}{16} }(x_0))}{(\frac{t}{16})^{d-1}}&\leq \exp(CN^\frac{2}{\beta-\frac{3}{4}}) + \frac{1}{4}\sup_{j}\frac{ H^{d-1}(Z(u_\e)\cap B_{t\tilde{s}_j}(x_0+ty_j))}{( t\tilde{s}_j)^{d-1}}.
	\label{itera1}
	\end{align}
Let $\tau=\frac{t}{16}$, then $C\e \exp(C(\ln N)^3)<\tau<\frac{5}{96}$.
Thus,
\begin{align}
	E_\e(x_0, \tau )\leq \exp(CN^\frac{2}{\beta-\frac{3}{4}}) +\frac{1}{4}\sup_j E_\e(\hat{y}_j, \hat{s}_j).
	\label{itera2}
	\end{align}
where $\hat{y}_j=x_0+ty_j\in B_\tau(x_0)$, $\hat{s}_j=t\tilde{s}_j\in (0, \frac{\tau}{8})$. Note that $B_{\hat{s}_j}(\hat{y}_j)$ may not be fully contained in
$B_\tau(x_0)$, since $ \hat{y}_j$ may be the centers of subcubes which intersect the boundary of $B_\tau(x_0)$ (we identify the ball $B_\tau(x_0)$ as a cube when we perform the subcubes decomposition).  However, $B_{\hat{s}_j}(\hat{y}_j)\subset B_{\tau+\frac{\tau}{8}}(x_0)$ since $\hat{s}_j\in (0, \frac{\tau}{8})$. If we iterate (\ref{itera2}), $B_{\hat{s}_j}(\hat{y}_j)$ still stays close to $B_{\tau}(x_0)$. Actually, $B_{\hat{s}_j}(y_j)\subset B_{\hat{\tau}}(x_0)$ for any large $j$, where $\hat{\tau}=\sum^\infty_{j=1} \frac{\tau}{8^{j-1}}=\frac{8\tau}{7}$.
	%By change of variable and translation, for  $\tau\in (0, \frac{1}{16}]$ and $B_{\tau}(y_0)\subset B_{\frac{1}{16}}(0)$,
	%\begin{align}
	%E_\e(y_0, \tau )\leq \exp(\exp(CN^\frac{2}{\beta-\frac{3}{4}})) +\frac{1}{4}\sup_j E_\e( y_j, \hat{s}_j).
	%\label{itera1}
	%\end{align}
	%for some $y_j\in B_{\tau}(y_0)$, $B_{\hat{s}_j}( y_j)\subset B_{\tau}(y_0)$, and $\hat{s}_j\in (0, \frac{\tau}{8})$.
	
Now, we iterate (\ref{itera2}) to obtain the desired estimate. The estimate (\ref{keyineq}) yields the initial step of the iteration,
\begin{align}
	E_\e(0, \frac{1}{16})\leq \exp(CN^\frac{2}{\beta-\frac{3}{4}})+\frac{1}{4}\sup_{1\le j \le \hat{m}} E_\e(y_j, \hat{t}_j).
	\end{align}
Assume that $\sup_{1\le j \le \hat{m}} E_\e(y_j, \hat{t}_j)$ is achieved at some  $E_\e(y_{j_0}, \hat{t}_{j_0})$ with $|y_{j_0}|<\frac{1}{16} $ and $|\hat{t}_{j_0}|<\frac{1}{128}$. Let $x_0=y_{j_0}$ and $\hat{t}_{j_0}=\tau$.
 Since $\hat{s}_j< \frac{\tau}{8}$, we apply (\ref{itera2}) to $E_\e(y_{j_0}, \hat{t}_{j_0})$ to get to the estimates of nodal sets at a smaller scale, that is,
 \begin{align}
	E_\e(0, \frac{1}{16})\leq (1+\frac{1}{4}) \exp(CN^\frac{2}{\beta-\frac{3}{4}})+\frac{1}{4} \sup_j E_\e(\hat{y}_j, \hat{s}_j).
	\end{align}
  We apply (\ref{itera2}) repeatedly down to the case $r\approx \hat{C}\e \exp(C(\ln N)^3)$ or the case that $E_\e(y,  r)$ is empty. Note that $B_{\hat{s}_j}(\hat{y}_j)\subset B_{\frac{1}{16}+\frac{8}{7\times 128 }} (0)\subset B_{\frac{1}{12}}(0)$.
	Thus, we derive that
	\begin{align}
	E_\e(0,  \frac{1}{16})&\leq \sum^{\infty}_{i=0} 4^{-i}\exp(CN^\frac{2}{\beta-\frac{3}{4}}) +\sup_{y\in B_{\frac{1}{12}}(0) } \big\{E_\e(y, r): \ 0<r\leq \hat{C}\e \exp(C(\ln N)^3) \big\} \nonumber \\
	&\leq \exp(CN^\frac{2}{\beta-\frac{3}{4}}) + (1+ \hat{C} \exp(C(\ln N)^3)) \exp(CN^\frac{2}{\beta-\frac{3}{4}})  \nonumber \\
	& \le \exp(CN^\frac{2}{\beta-\frac{3}{4}}),
	\label{combin1}
	\end{align}
	where we have used  (\ref{newnewdouble2}) in the second inequality. This proves the desired estimate for the case $\e\leq \exp(-C(\ln N)^3)$.
	
	Finally, for the case $\e \geq \exp(-C(\ln N)^3)$, the desired estimate follows directly from (\ref{newnewdouble2}). Since $\beta\in (\frac34,1)$ is arbitrary, so (\ref{est.Nodal}) holds for any $\alpha>8$.
	This ends the proof of the theorem.
\end{proof}

Following the above proof of Theorem \ref{mainth} for $d\ge 3$, we sketch the proof of upper bounds of nodal sets in $d=2$.

\begin{proof}[Proof of Theorem \ref{mainth} ($d = 2$)]
Since the proof is parallel to $d\ge 3$, we only present the changes for $d=2$. Thus,  we only present the changes for $d=2$.
By the argument in Lemma \ref{smallscale} and the doubling inequality (\ref{est.Br-2}), we can obtain for $0<r < 1/3$,
\begin{align}
	E_\e(x_0,r) \leq C\Big(1+ \frac{r}{\e}\Big) (\ln N)^2.
	\label{newnewdouble2-1}
	\end{align}
On the other hand, the statements of (\ref{set1}) and (\ref{est.sum.ti}) still hold for $u_0$; while for $\nabla u_0$, we have a better bound for $\hat{\delta}$.
Precisely for $d=2$, \cite[Theorem B.1]{NV17} implies
	\begin{equation}\label{est.Vol.BCr-2}
	|B_r(\mathcal{C}_r(u_0)) \cap B_s | \le C^{\big( \widetilde{N}(u_0,B_{2s} ) \big)} \Big( \frac{r}{s} \Big)^2 |B_s|.
	\end{equation}
By mimicking the argument in part (2) of Lemma \ref{strati}, we can show (\ref{set2}) and (\ref{est.sum.hattj}) with $0<\hat{\delta}< e^{-C \hat{N}^2_0}$. Following the proof of Lemma \ref{keylemma}, we may show that if $0<\e<\exp(-C(\ln N)^2)$, then
%and $\hat{\delta}\approx \exp(-C(\ln N)^2)$. Then we can reach the following iterative estimates
\begin{align}
	E_\e(0, \frac{1}{16})\leq  \exp(C(\ln N)^2)+\frac{1}{4}\sup_{1\le j \le \hat{m}} E_\e(y_j, \hat{t}_j),
	\label{keyineq-2}
	\end{align}
with $y_j\in B_{\frac{1}{16}}$, $\hat{t}_j \in (0, \frac{1}{128})$.
Observe that the quantitative stratification of critical sets $\nabla u_0$, instead of the doubling inequality, plays the dominant role in the estimate (\ref{keyineq-2}).
We iterate (\ref{keyineq-2}), as in the proof for $d\ge 3$, to get
\begin{align*}
	E_\e(0,  \frac{1}{16})&\leq \sum^{\infty}_{i=0} 4^{-i}\exp(C(\ln N)^2) +\sup_{y\in B_{\frac{1}{12}}(0) } \big\{E_\e(y, r): \ 0<r\leq \hat{C}\e \exp(C(\ln N)^2) \big\} \nonumber \\
	&\leq \exp(C(\ln N)^2) + C(1+ \hat{C} \exp(C(\ln N)^2)) (\ln N)^2  \nonumber \\
	& \le \exp(C(\ln N)^2).
	%\label{combin1-2}
	\end{align*}
This provides the desired estimate (\ref{est.Nodal2d}) for $0<\e<\exp(-C(\ln N)^2)$. For $\e\geq \exp(-C(\ln N)^2)$, (\ref{newnewdouble2-1}) yields the desired estimate directly. This ends the proof.
\end{proof}

%%\bibliographystyle{amsplain}
%%\bibliography{mybib}
\end{document}